\documentclass[10pt]{article}

\usepackage{amssymb,amsmath,amsfonts,verbatim, tikz, pgfplots, calc, amsbsy}
\usepackage{amsthm,amsbsy,parskip}
\usepackage{subfigure, graphicx, bookmark, array, algorithmic}
\usepackage[ruled, vlined]{algorithm2e}

\usepackage{epstopdf}

\usepackage[top=1in,bottom=1in,left=1.25in,right=1.25in]{geometry}

\newcommand{\R}{\mathbb{R}} %
\newcommand{\w}{\mathbf{w}} %
\newcommand{\y}{\mathbf{y}} %
\newcommand{\z}{\mathbf{z}} %
\newcommand{\A}{\mathbf{A}} %
\newcommand{\B}{\mathbf{B}} %
\newcommand{\X}{\mathbf{X}}
\newcommand{\postcov}{\boldsymbol{\Gamma}_\text{post}}
\newcommand{\prcov}{\boldsymbol{\Gamma}_\text{prior}}

\newcommand{\D}{\mathbf{D}}

\newcommand{\bOmega}{\boldsymbol\Omega}

\newcommand{\Ev}{\mathbb{E}} %
\newcommand{\Var}{\mathrm{Var}} %
\newcommand{\mP}{\mathbb{P}} %
\newcommand{\calE}{\mathcal{E}} %
\newcommand{\calD}{\mathcal{D}} %

\newcommand{\tr}{\mathrm{tr}} %
\newcommand{\Lmax}{\lambda_{max}} %
\newcommand{\Lmin}{\lambda_{min}} %
\newcommand{\ds}{\displaystyle}
\newcommand{\BEA}{\begin{eqnarray*}}
\newcommand{\EEA}{\end{eqnarray*}}
\newcommand{\normP}[1]{{\left\vert\kern-0.25ex\left\vert\kern-0.25ex\left\vert #1    \right\vert\kern-0.25ex\right\vert\kern-0.25ex\right\vert}_p}
\usepackage{color}

\newtheorem{definition}{Definition} %
\newtheorem{proposition}{Proposition} %
\newtheorem{theorem}{Theorem} %
\newtheorem{remark}{Remark} %

\title{Monte Carlo Estimators for the Schatten \texorpdfstring{$p$}{p}-norm of Symmetric Positive Semidefinite  Matrices}

\author{Ethan Dudley\footnotemark[2], 
        Arvind K. Saibaba\footnotemark[2], Alen Alexanderian\footnotemark[2]}

\begin{document}

\maketitle

\renewcommand{\thefootnote}{\fnsymbol{footnote}}

\footnotetext[2]{Mathematics Department, North Carolina State University, North Carolina, United States}

\begin{abstract}
   We present numerical methods for computing 
    the Schatten $p$-norm of positive semi-definite matrices. Our motivation stems from uncertainty quantification and optimal experimental design for inverse problems, where the Schatten $p$-norm defines a design criterion known as the P-optimal criterion. Computing the Schatten $p$-norm of high-dimensional matrices is computationally expensive. We propose a matrix-free method to estimate the Schatten $p$-norm using a Monte Carlo estimator and derive convergence results and error estimates for the estimator. To efficiently compute the Schatten $p$-norm for non-integer and large values of $p$, we use an estimator using a Chebyshev polynomial approximation and extend our convergence and error analysis to this setting as well. We demonstrate the performance of our proposed estimators on several test matrices and through an application to optimal experimental design of a model inverse problem.
\end{abstract}

\section{Introduction}\label{intro}
    The Schatten $p$-norm of a  matrix $\A \in \R^{m\times n}$ is defined as
    \[
        \normP{\A} = \left(\sum_{j = 1}^{\min\{m,n\}} \sigma_j^p\right)^{1/p}
    \]
    where $p \geq 1$ and $\sigma_j$ is the $j$th singular value of $\A$ for $1 \leq j \leq \min\{m,n\}$. 
    If $\A \in \R^{n \times n}$ is a symmetric positive semi-definite (SPSD) matrix, then the singular values of $\A$ are its eigenvalues, and the Schatten $p$-norm takes the form
    \begin{equation}\label{eq: Schatten p-norm def}
       \normP{\A} = \Big(\sum_{j=1}^n \lambda_j^p\Big)^{1/p} = \big(\tr(\A^p)\big)^{1/p}
    \end{equation}
    where the $\lambda_j$'s are the eigenvalues of $\A$. There are several notable
    special cases of the Schatten $p$-norm including the nuclear norm ($p = 1$), the Frobenius norm ($p = 2$) and the spectral norm ($p \to \infty$). Since it encapsulates many well-known norms as special cases, the Schatten $p$-norm is frequently used in linear algebra and analysis~\cite{bhatia2013matrix}. 
    
     Our motivation for computing the Schatten $p$-norm arises from uncertainty quantification
     and optimal experimental design (OED) for Bayesian inverse problems. An inverse problems seeks to estimate parameters of interest using experimental measurements. The goal of OED is to identify an optimal 
     set of experiments by optimizing certain design criteria that measure the uncertainty in the estimated parameters, subject to budgetary or physical constraints. 
     A well-known design criterion, known as the P-optimal design criterion, can be expressed in terms of the Schatten-$p$ norm. Since optimization algorithms for OED require repeated evaluations of the design criterion for large matrices, efficient algorithms for estimating the Schatten-p norm are desirable. 
    
    In this article, we focus on computing the Schatten-p norm for large SPSD matrices. For such matrices computing the Schatten $p$-norm is computationally challenging, because it requires computing either the matrix $p$th power or all of its eigenvalues. However, if the matrix is large and its entries are not available explicitly, then the Schatten $p$-norm cannot be easily computed from its definition \eqref{eq: Schatten p-norm def}, and special numerical methods are necessary. Therefore, we consider computing the Schatten $p$-norm using matrix-free Monte Carlo methods. In a matrix-free method for computing $\normP{\A}$ we only require matrix-vector products involving $\A$. 
    \textbf{Related work}.
    Hutchinson~\cite{hutchinson1990stochastic} developed a matrix-free Monte Carlo estimator using samples from the Rademacher distribution for computing $\tr(\A)$, i.e., the Schatten-1 norm. %
    Avron and Toledo~\cite{avron2011randomized} extended this idea to random variables from other distributions such as Gaussian and uniformly selected vectors from an orthogonal matrix. They devised several metrics for comparing the various trace estimators including a single sample variance metric and a Chernoff-style lower bound on the minimum number of samples required to meet a given error tolerance with a given confidence level. 
    This is made precise in the following definition:   
    \begin{definition}\label{def: var-delta est}
    Given $\varepsilon > 0$ and $\delta \in (0,1)$ and an appropriate distribution of random samples $\w_j \in \R^n$, we say 
    \[
        Z_M = \frac{1}{M}\sum_{j=1}^M\w_j^T\A\w_j
    \]
    is an $(\varepsilon, \delta)$ estimator for $\tr(\A)$ if 
    \begin{equation}\label{eq: eps-delta Estimator}
        \mP\left(\big|Z_M - \tr(\A)\big|\leq \varepsilon|\tr(\A)|\right)\geq 1 - \delta.
    \end{equation}
    \end{definition}
    This definition alternatively says that $Z_M$ is an $(\varepsilon,\delta)$ estimator if with probability at least $1-\delta$, it has a relative error at most $\varepsilon$. Avron and Toledo~\cite{avron2011randomized}  provided a lower bound on the number of samples so that $Z_M$ is $(\varepsilon, \delta)$ estimator for $\tr(\A)$ when $\w_j$ are drawn from the Gaussian, Rademacher and Uniform distributions. Roosta-Khorasani and Ascher~\cite{roosta2015improved} further reduced the lower bound on the number of samples needed for an $(\varepsilon, \delta)$ estimator for $\tr(\A)$ when the estimators use random vectors from the Rademacher and Gaussian distributions. This Monte Carlo estimator has been extended to Schatten-p norm using Chebyshev polynomials~\cite{han2017approximating} and Lanczos approach~\cite{ubaru2017fast}.
    
    A recent survey paper by Martinsson and Tropp~\cite{martinsson2020randomized} reviews estimators for the Schatten-p norms, which avoid working with $\A^p$ directly.  Let $\X = \bOmega^T\A\bOmega$, where the entries of $\bOmega \in \mathbb{R}^{n\times M}$ have zero mean and unit variance. The estimator $V_p$ in Kong and Valiant~\cite{kong2017spectrum}, is 
    \[ V_p = \begin{pmatrix} M \\ p\end{pmatrix}^{-1} \tr(\mathcal{T}(\X)^{p-1}\X),\]
    where $\mathcal{T}(\X)$ is a matrix that contains the strictly upper triangle part of $\X$, and zeroes out the rest of the entries. Note that $V_p$ is an unbiased estimator for $\normP{\A}$. A related estimator is 
    \[ W_p = \frac{(M - p)!}{M!} \sum_{1\leq i_1,\dots,i_p\leq M} \X_{i_1,i_2}\X_{i_2,i_3}\dots\X_{i_p,i_1},\]
    where the summation is only over distinct indices. Similar to $V_p$, $W_p$ is an unbiased estimator for $\normP{\A}$. For both estimators, the recommended number of samples $M \gtrsim n^{1-2/p}.$ This lower bound was established by~\cite{li2014sketching}. Both of these estimators are expensive for large $p$; however, the algorithm only requires $M$ matrix-vector products involving $\A$. Theoretical analysis suggests that the variance of these estimators are large which makes their use for large-scale applications impractical~\cite{martinsson2020randomized}.

    \textbf{Our approach and contributions}.
    We focus on analysis and efficient computational methods for the following estimator of $\normP{\A}$  
    \[ 
    \normP{\A} \approx \left(\frac{1}{M}\sum_{j=1}^M\w_j^T\A^p\w_j , \right)^{1/p},
    \]
    where $\w_j$ are random vectors from an appropriate distribution. 
    To our knowledge, an analysis of the convergence of this (biased) estimator has not been performed in the literature. Computing the Monte Carlo estimator involves repeated applications of $\A^p$ to a vector, which is computationally expensive for large or non-integer values of $p$. To reduce this cost, two different approaches were proposed based on Chebyshev polynomial approximation~\cite{han2017approximating} and based on Lanczos approach~\cite{ubaru2017fast}. In this article, building on the work~\cite{han2017approximating}, we consider approximate Monte Carlo estimators based on Chebyshev polynomials.

    The following are the main contributions of this article.
    \begin{enumerate}
        \item In our analysis of the new estimator we derive bounds on the expectation, bias and variance (Section~\ref{sec: moments}) and we show the estimator converges almost surely as well as in $L^1$ and $L^2$ (Sections~\ref{sec: BuildEst} and~\ref{sec: converge}). In Section~\ref{sec: eps-delta}, we show that the number of samples required to form an $(\varepsilon, \delta)$ estimator for $\normP{\A}$ does not grow as $p \to \infty$.
        \item In section~\ref{sec: Cheby}, we consider a variation of the Chebyshev-Monte Carlo method proposed by~\cite{han2017approximating}. This approach is applicable to non-integer values as well as large values of $p$. We extend our results from the standard Monte Carlo approach to the Chebyshev-Monte Carlo approach.
        \item We provide extensive numerical tests on synthetic matrices, matrices arising from real-world problems and a model problem from OED which help illustrate the theoretical results. We also provide numerical evidence that a small degree Chebyshev approximation $\psi_N(\A)$ to $\A^{p/2}$ is sufficient for an accurate estimator. 
    \end{enumerate}

\section{Background}\label{sec: background}
In this section, we review known results for the two largest contributing ideas in this article: Monte Carlo Trace Estimators (Section~\ref{ssec:mctrace}) and Chebyshev Polynomials (Section~\ref{ssec:chebyshev}).

\subsection{Monte Carlo Trace Estimators}\label{ssec:mctrace}
    Let $(\Omega, \mathcal{F}, \mP)$ be a probability space. 
    \begin{definition}\label{def: trace est}
    Let $\w: \Omega \to \R^n$ be a random $n$-vector with mean 0 and identity covariance matrix, and $\B$ be a symmetric matrix. Then the Monte Carlo trace estimator of $\B$ is given by
    \begin{equation}\label{eq: ZM}
        Z_M = \frac{1}{M}\sum_{j = 1}^M\w_j^T\B\w_j,
    \end{equation}
    where $\w_j$, $j=1, \ldots, M$ are distributed according to the law of $\w$.
    \end{definition}
    
    We call $Z_M$ a trace estimator of $\B$ because $\Ev(\w^T\B\w) = \tr(\B)$ and therefore by the linearity of expectation $\Ev(Z_M) = \tr(\B)$~\cite{hutchinson1990stochastic,avron2011randomized}. Furthermore, since $\w_j^T\B\w_j \in L^1(\Omega,\mathcal{F}, \mP)$, by the strong law of large numbers~\cite{jacod2012probability}, we have
    \[
        \mP\left( \lim_{M \to \infty} Z_M = \tr(\B)\right)= 1.
    \]
    That is, $Z_M$ converges to the $\tr(\B)$ almost surely (a.s.). Lastly, we can formulate a Chernoff-style lower bound on $M$ to guarantee that $Z_M$ is an $(\varepsilon, \delta)$ estimator; that is, it is the least number of samples to guarantee $Z_M$ is an $(\varepsilon, \delta)$ estimator for $\tr(\B)$ (i.e., $Z_M$ satisfies Definition~\eqref{def: var-delta est}). Note that the $(\varepsilon, \delta)$ bound on $M$ is dependent on the distribution from which the $\w_j$ are chosen, as the different distributions affect the variance of $Z_M$. This is summarized in Table~\ref{tab: ZM properties}. 
    \begin{table}[!ht]
        \centering
        {\renewcommand{\arraystretch}{2}%
        \begin{tabular}{|c|c|c|}
            \hline
            & $\Var(Z_M)$ & $(\varepsilon, \delta)$ bound \\
            \hline
            \text{Gaussian} & $\frac{2\|\B\|_F^2}{M}$ & $M \geq 8\varepsilon^{-2}\ln\left(\frac{2}{\delta}\right)$\\
            \hline 
            \text{Rademacher} & $\frac{2(\|\B\|_F^2 - \sum_{i = 1}^n\B_{ii}^2)}{M}$ & $M \geq 6\varepsilon^{-2}\ln\left(\frac{2}{\delta}\right)$\\
            \hline 
        \end{tabular}}
        \caption{The Variance and the number of samples required for $(\varepsilon, \delta)$ bound for $\tr(\B)$. Here, $\w_j$'s are chosen from the Gaussian and Rademacher distributions~\cite{avron2011randomized, roosta2015improved}}
        \label{tab: ZM properties}
    \end{table}

    \subsection{Chebyshev polynomials}\label{ssec:chebyshev}
    Throughout this article, we will use Chebyshev polynomials of the first kind, which are defined as 
    \[
        T_j(x) = \cos(j\arccos(x)) \qquad x \in [-1,1] \quad j = 0, 1, 2, \dots 
    \]
    As is well-known, these polynomials are orthogonal with respect to the inner product 
    $\langle u, v\rangle_w = 
    \int_{-1}^1 u(x) v(x) \, w(x)dx$, with the weight function $w(x) = 1/\sqrt{1-x^2}$, and 
    \[
        \langle T_i, T_j\rangle_w = 
        \begin{cases}
        \pi & i = j = 0,\\
        \frac{\pi}{2} & i = j \neq 0,\\
        0 & i \neq j.
        \end{cases}
    \]
    Moreover, any continuous function $g$ on the interval $[-1,1]$ can be expressed as~\cite{boyd2001chebyshev}
    \[
        g(x) = c_0 T_0(x) + \sum_{j = 1}^\infty c_jT_j(x),
    \]
    where 
    the series converges uniformly and the coefficients can be computed as 
    \begin{equation}\label{eq: Cheby Coeff}
        c_j = \ds\frac{2}{\pi}\int_{-1}^1 \frac{g(x)T_j(x)}{\sqrt{1 - x^2}}\; dx = \ds\frac{2}{\pi}\int_{0}^\pi g(\cos(\theta))\cos(j\theta)\; d\theta
    \end{equation}
    and $c_0$ carries an additional factor of a half. Note that the Chebyshev polynomial approximation to a function is equivalent to the Fourier cosine series approximation of $g$,~\cite{boyd2001chebyshev} therefore the coefficients $c_j$ can be computed using the real part of the Fast Fourier Transform (FFT) of $g$; see~\cite{trefethen2008gauss}
    for details, as well as computer code for doing so.
    
    Let $\psi_N(x)$ be the $N$th degree Chebyshev approximation to $g$. The error in $\psi_N(x)$ is bounded tightly by~\cite{boyd2001chebyshev}
    \[
        \left|g(x) - \psi_N(x)\right| \leq \sum_{j = N+1}^\infty |c_j|.
    \]
    Trefethen~\cite{trefethen2008gauss} presented a method for approximating this error without computing the remaining coefficients for both analytic functions and functions with singularities in the complex plane. Here we present the analytic version:
    \begin{equation}\label{eq: Trefethen cheby err bound}
        \left|g(x) - \psi_N(x)\right| \leq \frac{4U}{\rho^N(1-\rho)},
    \end{equation}
    where $g$ is analytic on the inside of an ellipse $E$ in the complex plane with foci at $\pm 1$, $U = \ds\sup_{z \in E}g(z)$ and $\rho$ is the sum of the major and minor semi-axes of $E$ with $\rho > 1$. 

    Finally, we recall that Chebyshev polynomials have a three term recurrence relation~\cite{boyd2001chebyshev}:
    \[
        T_{j+1}(x) = 2xT_{j}(x) - T_{j-1}(x), \quad x \in [-1,1],
    \]
    with $T_0(x) = 1$ and $T_1(x) = x$. 
    This ensures that matrix-vector products using the Chebyshev matrix polynomials can be computed in a matrix-free manner, which is useful in constructing a Monte Carlo approximation to $\normP{A}$.

\section{Monte Carlo Estimators and their Analyses}\label{sec: mc}
    In this section, we construct a Monte Carlo estimator for the Schatten $p$-norm (Section \ref{sec: BuildEst}) and present a detailed
    analysis of convergence of the estimator (Section \ref{sec: converge}). 
    \subsection{Building a Schatten \texorpdfstring{$p$}{p}-norm Estimator}\label{sec: BuildEst}
        Recall if $\A$ is SPSD and $\w:\Omega \to \R^n$ is an $n$-vector with mean $0$ and identity covariance matrix, then
        \[
            \normP{\A}^p = \tr(\A^p) = \Ev(\w^T\A^p\w).
        \]
        Therefore, consider the following Monte Carlo estimator for $\normP{\A}$.
        \begin{definition}\label{def:X_M}
            Let $\A$ be an SPSD matrix. We define the Monte Carlo estimator for $\normP{\A}$ as  
            \begin{equation}\label{eq: X_M}
            X_M = \left(\frac{1}{M}\sum_{j = 1}^M\w_j^T\A^p\w_j\right)^{1/p}, \quad M \geq 1,
        \end{equation}
        where  $\w_j$'s are realizations of a random variable $\w : \Omega \to \R^n$ with $\Ev(\w) = 0$ and $\Ev(\w \w^T) = \mathbf{I}$. 
        \end{definition}

        Note that $X_M^p$ is an unbiased estimator for $\normP{\A}^p$. Furthermore, if $p = 1$, then $X_M$ is just the Monte Carlo trace estimator~\eqref{eq: ZM}.
        \begin{algorithm}[!ht]
            \caption{Constructing the Monte Carlo Estimator $X_M$}\label{algorithm:: mc}
            \begin{algorithmic}
	            \STATE \textbf{Input:} a SPSD matrix $\A \in \R^{n\times n}$, positive integers $M$ (number of samples) and $p$ (Schatten-p degree)
	            \STATE \textbf{Initialize:} $X_M \gets 0$
	            \STATE $K \gets \ds\Big\lfloor\frac{p}{2}\Big\rfloor$
	            \FOR{$j = 1$ \text{ to } $M$} 
	                \STATE $\w_j \gets$ random vector with mean \textbf{0} and covariance $\mathbf{I}$
	                \STATE $\y \gets \A^K\w_j$
	                \STATE \textbf{If }{$p$ is odd} 
	                    \STATE $X_M \gets X_M + \y^T\A\y/M$
	                \STATE \textbf{Else}
	                    \STATE $X_M \gets X_M + \y^T\y/M$
	                \STATE \textbf{EndIf}
	            \ENDFOR
	            \STATE $X_M \gets (X_M)^{1/p}$
            \end{algorithmic}
        \end{algorithm}
        
        In Algorithm \ref{algorithm:: mc}, we provide a pseudo-code for efficiently computing $X_M$ for positive integer values of $p$. First, note that by using the symmetry of $\A$, computing $X_M$ using Algorithm~\ref{algorithm:: mc} requires $\left\lceil\frac{p}{2}\right\rceil M$ matrix-vector products with $\A$. Second, the algorithm is general in the sense that any distribution for the random vectors $\w_j$ can be used so long as $\w_j$'s are independent and drawn from a distribution that has mean zero and the identity matrix as its covariance. However, in our analysis, we assume that the entries of
        $\w_j$'s are independent standard normal random variables. If a different distribution is used, then the number of samples required for an $(\varepsilon,\delta)$ estimator for $\normP{\A}$ will have to be changed appropriately. 

         We first collect a series of results for the estimator $X_M^p$ in Proposition~\ref{prop:mc}.  Then, in the rest of this section we appropriately adapt these results to the estimator $X_M$. 
         
        \begin{proposition}\label{prop:mc}
           The estimator $X_M$ satisfies the following properties:
            \begin{enumerate}
                \item (Expectation): $\Ev(X_M^p) = \normP{\A}^p$.\label{prop:Ev X_M^p}
                \item (Variance): $\Var(X_M^p) = \displaystyle\frac{2\|\A^p\|_F^2}{M}$.\label{prop:Var X_M^p}
                
                \item (Almost Sure Convergence): $\displaystyle\lim_{M \to \infty}X_M^p = \normP{\A}^p$ a.s.\label{prop:X_M^p conv a.s.}
                \item ($(\varepsilon, \delta)$ Estimator): If $M \geq 8\varepsilon^{-2}\ln\left(\displaystyle\frac{2}{\delta}\right)$ then $X_M^p$ is an $(\varepsilon, \delta)$ estimator for $\normP{\A}^p$.\label{prop:eps-delta bound X_M^p}
                \item (Non-negative): $X_M^p \geq 0$ for all $M$.\label{prop:X_M^p nonneg}
            \end{enumerate}
        \end{proposition}
        
        \begin{proof} %
            The proof collects well-known results from the literature. 
            The expressions for the expectation and variance of $X_M^p$ follow 
            from~\cite[Lemma 5]{avron2011randomized}.
            To see the third statement, note that since $X_M^p \in L^2(\Omega, \mathcal{F}, \mP)$, by the strong law of large  numbers~\cite{jacod2012probability} we have
            \[
                \displaystyle\lim_{M \to \infty}X_M^p = \Ev(X_M^p) = \normP{\A}^p \qquad \text{a.s.}
            \]
            Regarding the fourth statement, Roosta-Khorasani and Ascher~\cite{roosta2015improved} showed that if $\varepsilon > 0$, $\delta \in (0,1)$ and 
            $M \geq 8\varepsilon^{-2}\ln({2}/{\delta})$,
            then $X_M^p$ is a $(\varepsilon, \delta)$ estimator of $\normP{\A}^p$. %
            Finally, since $\A$ is SPSD, then so is $\A^p$. Thus, $\w_j^T\A^p\w_j \geq 0$ for all $j$. Hence $X_M^p \geq 0$ for all $M$. \hfill
        \end{proof}
        
    Since the quantity of interest is $\normP{\A}$, we have to analyze its estimator $X_M$. While Proposition~\ref{prop:mc} states several properties for $X_M^p$, a natural question is the extent to which these properties apply to $X_M$. We investigate this in the rest of this section. 
        
    \subsection{Expectation and Variance of \texorpdfstring{$X_M$}{}}\label{sec: moments}
        In this section we will provide a bound on the first two moments of $X_M$. Specifically we show that $X_M$ is biased for all finite values of $M$ and we provide an upper bound on the variance of $X_M$.
        
        \begin{proposition}[Expectation]\label{proposition: X_M EV}
           For all $M \geq 1$, $\Ev\left(X_M\right) \leq \normP{\A}$. 
        \end{proposition}
        \begin{proof}
            Let $M$ be any natural number. Since $X_M = (X_M^p)^{1/p}$ and $f(x) = x^{1/p}$ is concave, by Jensen's inequality~\cite{jacod2012probability}, we have
            $\Ev(X_M) = \Ev\left((X_M^p)^{1/p}\right) \leq \left(\Ev(X_M^p)\right)^{1/p} = \normP{\A}$.
        \end{proof}
        
        Note that, in the above result, equality is attained for $p=1$. Similar to the first moment we will derive an upper bound on the variance of $X_M$. 
        
        \begin{proposition}[Variance]\label{prop: X_M Var}
        If $\A$ is nonzero, then the variance in $X_M$ is finite and satisfies
        \[
            \Var\left(X_M\right) \leq \frac{2\|\A^p\|_F^2}{M\normP{\A}^{2p-2}}.
        \]
        \end{proposition}
        \begin{proof}
            Without loss of generality, assume  $p > 1$, otherwise the variance bound holds trivially. By ~\cite[Theorem 1, Corollary 1]{nollau1995inequalities}, if $Y$ is a non-negative random variable with positive mean and finite variance, then for $\alpha \in [0,1]$
            \begin{equation}\label{eqn:vary} \Var(Y^{\alpha}) \leq \Ev |Y^\alpha - (\Ev Y)^\alpha|^2  \leq \frac{\Var(Y)}{(\Ev Y)^{2-2\alpha}}.\end{equation}
            We let $Y = X_M^p$ and $\alpha = 1/p$. Note that $Y$ is non-negative, $\Ev (Y)  = \normP{\A}^p>  0$ since $\A$ is nonzero, and from Table~\ref{tab: ZM properties}, $\Var(Y) = 2 \|\A^p\|_F^2/M < \infty$. Therefore,~\eqref{eqn:vary} applies, and 
            \[ \Var(X_M) = \Var\left((X_M^p)^{1/p}\right) \leq \frac{2\|\A^p\|_F^2}{M\normP{\A}^{p(2-2/p)} } 
            = \frac{2\|\A^p\|_F^2}{M\normP{\A}^{2p-2} }. \qedhere %
            \]
        \end{proof}
        
        \subsection{Convergence of estimators}\label{sec: converge}
        In this section, we show that $X_M$ converges almost surely, in $L^1$, and in $L^2$ as $M \to \infty$. %

       \begin{proposition}[Almost sure convergence]\label{prop: X_M a.s.}
       We have 
             $\ds\lim_{M \to \infty}X_M = \normP{\A}$ almost surely.
        \end{proposition}
        \begin{proof}
            Let $f(x) = x^{1/p}$. Note that $f$ is continuous for $x \geq 0$. Since $\A$ is SPSD, from Proposition~\ref{prop:mc}, $X_M^p$ is non-negative and converges almost surely to $\normP{\A}^p$. Thus we can apply the Continuous Mapping Theorem~\cite[Theorem 17.5]{jacod2012probability} to obtain
            \[
                \lim_{M \to \infty}X_M = \lim_{M \to \infty}\left(X_M^p\right)^{1/p} = \left(\normP{\A}^p\right)^{1/p} = \normP{\A} \text{ a.s.} \qedhere %
            \]
        \end{proof}

        Recall that by Proposition~\ref{proposition: X_M EV}, $\Ev(X_M) \leq \normP{\A}$. Now, we form a bound on the bias in $X_M$. This will also be useful for establishing convergence in $L^1$ and in $L^2$.
        \begin{proposition}[Bias]\label{prop: Bias bound}
        The bias in $X_M$ is bounded as 
        \[
            |\Ev(X_M) - \normP{\A}| \leq {\normP{\A}}{}\left(\frac{2}{M}\right)^{1/2}.
        \]
        \end{proposition}
        \begin{proof}
        Without any loss in generality, assume that $\A$ is nonzero so that $\normP{\A} \neq 0$. Similarly, assume that $p > 1$ otherwise the bias is zero and the bound holds trivially.  %
        As $\A$ is SPSD, by Proposition~\ref{prop:mc}, $X_M \geq 0$ and $\Ev(X_M)$ is the $L^1$ norm of $X_M$. Then, by the reverse triangle inequality and the Cauchy-Schwarz inequality
        \begin{equation}\label{eqn:biasbound}
        \begin{aligned}
            \left|\Ev(|X_M|) - \normP{\A}\right| \leq & \> \Ev\left(\left|X_M - \normP{\A}\right|\right) \leq \left(\Ev\left|X_M - \normP{\A}\right|^2\right)^{1/2}  \\\leq & \> \left(\frac{2\|\A^p\|_F^2}{M\normP{\A}^{2p-2}}\right)^{1/2}.
            \end{aligned}
        \end{equation}
        The last inequality follows from \eqref{eqn:vary} with $Y = X_M^p$ and $\alpha = 1/p$. 
        Now using the fact that $\A$ is SPSD 
        \[ 
            \normP{\A}^{2p} = \Big(\ds\sum_{j=1}^n\lambda_j^p\Big)^{2} \geq \sum_{j=1}^n\lambda_j^{2p}  = \|\A^p\|_F^2. 
        \] 
        Therefore, we have 
        \[ \frac{\|\A^p\|_F^2}{ \normP{\A}^{2p-2}}  = \normP{\A}^2\frac{\|\A^p\|_F^2}{\normP{\A}^{2p}} \leq  \normP{\A}^2. \]
        Substitute this into~\eqref{eqn:biasbound} and simplify to obtain the desired inequality. \hfill\end{proof}

        Proposition~\ref{prop: Bias bound} can be readily used to establish $L^1$ convergence. For a fixed $p$, $X_M$ converges in $L^1(\Omega, \mathcal{F}, \mP)$ to $\normP{\A}$ since 
        \[\lim_{M \to \infty}\left|\Ev(X_M) - \normP{\A}\right| \leq \lim_{M \to \infty} {\normP{\A}}{}\left(\frac{2}{M}\right)^{1/2} = 0. \]
        Similarly, convergence in $L^2(\Omega,\mathcal{F},\mP)$ follows from the proof of Proposition~\ref{prop: Bias bound}.

  \subsection{Number of Samples for an \texorpdfstring{$(\varepsilon, \delta)$}{} Estimator of \texorpdfstring{$\normP{\A}$}{}}\label{sec: eps-delta}
        In this section we determine the minimum number of samples required to form an $(\varepsilon, \delta)$ estimator for $\normP{\A}$.         
        \begin{theorem}[$(\varepsilon,\delta)$ estimator]\label{theorem: ProbConc}
            For all $\varepsilon > 0$ and $\delta \in (0,1)$, the number of samples required for $X_M$ to be an $(\varepsilon, \delta)$ estimator for $\normP{\A}$ satisfies
            \begin{equation}\label{eq: M bound}
                M \geq 8\varepsilon^{-2}\ln\left(\frac{2}{\delta}\right)
            \end{equation}
        \end{theorem}
        \begin{proof}
            Consider the measurable sets 
            \[
                \calD = \Big\{\omega \in \Omega : \big|X_M(\omega) - \normP{\A}\big|\leq \varepsilon\normP{\A}\Big\}
           \quad \text{and} \quad
                \calE = \Big\{\omega \in \Omega : \big|X_M^p(\omega) - \normP{\A}^p\big| \leq \varepsilon\normP{\A}^p\Big\}.
            \]
            Note that if $\A$ is the zero matrix then both of these events are equivalent and have probability 1. Now consider when $\A$ is a non-zero SPSD matrix.
            Roosta-Khorasani and Ascher~\cite[Theorem 3]{roosta2015improved} showed that for $\varepsilon,\delta$ as in the statement of the theorem, $\mP(\calE) \geq 1 - \delta$, if 
            \[
                M \geq 8\varepsilon^{-2}\ln\left(\frac{2}{\delta}\right).
            \]
            Thus, it is sufficient to show that $\calE \subset \calD$. Therefore, consider when $\omega \in \calE$. One can show using the difference of powers formula, that $f(x) = x^{1/p}$ satisfies 
            \[
                \Big|(x+h)^{1/p} - x^{1/p}\Big| \leq \frac{|h|}{x^{1-1/p}}
            \]
            for all $x > 0$ and $h \geq -x$.  Since $\A$ is nonzero we let $x = \normP{\A}^p$ and $h = X_M^p(\omega) - \normP{\A}^p$. Then 
            \[
                \Big|X_M(\omega) - \normP{\A}\Big| 
                 \leq  \frac{|X_M^p(\omega) - \normP{\A}^p|}{\normP{\A}^{p-1}}
                 \leq  \frac{\varepsilon\normP{\A}^p}{\normP{\A}^{p-1}}
                 =  \varepsilon\normP{\A}.
            \]
            Thus, $\omega \in \calD$. Therefore, $\calE \subset \calD$. Therefore, 
            \[
                1 - \delta \leq \mP(\calE) \leq \mP(\calD).  \qedhere %
            \]
        \end{proof}
        
        As a result of this theorem we have the following important remark:
        \begin{remark}
            For all $p \geq 1$, the minimum number of samples required for $X_M$ to become an $(\varepsilon, \delta)$ estimator for $\normP{\A}$ is independent of $p$.  
        \end{remark} %

    In contrast, numerical evidence suggests that the empirical variance decreases with increasing $p$, suggesting that the number of samples should correspondingly decrease for the same level of accuracy. However, we were not able to derive a theoretical result justifying this observation.

\section{Chebyshev Monte Carlo estimator and its analysis}\label{sec: Cheby}
Recall that Algorithm~\ref{algorithm:: mc} requires $O\left(\left\lceil\frac{p}{2}\right\rceil M\right)$ matrix-vector products and can be computationally expensive for large $p$; similarly, the Algorithm is not applicable to non-integer values of $p$. To address this issue, we use a Chebyshev polynomial approximation to approximate $\A^p$ by a lower degree Chebyshev polynomial $\psi_N(\A)$. A similar approach was used in~\cite{han2017approximating} in the context of estimating the trace of matrix functions. In this section, we propose a new estimator for the Schatten-p norm and extend our analysis of convergence on the standard Monte Carlo estimator to the estimator using Chebyshev polynomial approximation. In contrast to the previous section, where it was sufficient for $\A$ to be SPSD, in this section, we require $\A$ to be SPD.

\subsection{Chebyshev Polynomial Approximation Method}
Recall that the $N$th degree Chebyshev polynomial approximation of a continuous function $g(x)$ with $x \in [\Lmin, \Lmax]$, contained in the interval $[a,b]$,  $0 < a \leq \Lmin \leq \Lmax \leq b$, is given by
    \[
        g(x) \approx \psi_N(x) = c_0 + \sum_{j = 1}^N c_jT_j\left(\frac{2}{\Lmax- \Lmin}x + \frac{\Lmax + \Lmin}{\Lmax - \Lmin}\right)
   \]
where $T_j(x) = \cos(j\arccos(x))$ is the $j$th Chebyshev polynomial, and the coefficient $c_j$ is defined in \eqref{eq: Cheby Coeff}. In this article, since we are computing the Schatten p-norm, the function of interest is $g(x) = x^{p/2} \approx \psi_N(x)$. Based on this polynomial approximation, we can construct the Chebyshev polynomial approximation to $\A^p \approx [\psi_N(\A)]^2$. This ensures that the Chebyshev polynomial approximation to $\A^p$ is symmetric positive semidefinite. This is an important point since approximating $\A^p \approx \psi_{N'}(\A)$ using Chebyshev polynomials, does not automatically guarantee semidefiniteness.
            
In Algorithm~\ref{algorithm:: Cheby} we present an efficient algorithm for approximating $\normP{\A}$ using the Chebyshev-Monte Carlo method, based on the discussion in~\cite{han2017approximating,trefethen2008gauss}. The method combines the Chebyshev polynomial approximation for $x^{p/2}$ in $[\lambda_\text{min},\lambda_\text{max}]$ along with the three-term recurrence property of the Chebyshev polynomials. For Algorithm~\ref{algorithm:: Cheby} to be cost effective compared to Algorithm~\ref{algorithm:: mc}, the degree of the Chebyshev approximation should satisfy $N < \ds\frac{p}{2}$. Furthermore, observe that Algorithm~\ref{algorithm:: mc} requires at least a crude estimate $[a,b]$ of the range for the spectrum of $\A$. This can be accomplished using matrix free methods such as Krylov subspace methods~\cite{saad2011numerical}. In our implementation, we use the MATLAB command \verb|eigs|. 

 \begin{algorithm}[!ht]
    \caption{Constructing the Monte Carlo Estimator $Y_{M,N}$}\label{algorithm:: Cheby}
    \begin{algorithmic}
	   \STATE \textbf{Input:} a SPD matrix $\A \in \R^{n\times n}$ with eigenvalues in [$a, b$], sample number $M$, a Chebyshev polynomial degree $N$ and Schatten $p$-norm degree $p$
	   \STATE \textbf{Initialize:} $Y_{M,N} \gets 0$
	   \STATE $c \gets$ N+1 vector of Chebyshev Coefficients for $x^{p/2}$ (see \eqref{eq: Cheby Coeff})
	   \FOR {$j = 1$ \textbf{to} $M$} 
	       \STATE $\w_j \gets$ random vector with mean \textbf{0} and covariance $\mathbf{I}$
	       \STATE $\y_0^{(j)} \gets \w_j$ and $\y_1^{(j)} \gets \ds\frac{2}{b - a}\A\w_j - \ds\frac{b + a}{b - a}\w_j$
	       \STATE $\z \gets c_0\y_0^{(j)} + c_1\y_1^{(j)}$ 
	       \FOR {$k = 2$ \textbf{to} $N$}
	           \STATE $\y_2^{(j)} \gets \ds\frac{4}{b - a}\A\y_1^{(j)} - \ds\frac{2(b + a)}{b - a}\y_1^{(j)} - \y_0^{(j)}$
	           \STATE $\z \gets \z + c_k\y_2^{(j)}$
	           \STATE $\y_0^{(j)} \gets \y_1^{(j)}$ and $\y_1^{(j)} \gets \y_2^{(j)}$
	       \ENDFOR
	       \STATE $Y_{M,N} \gets Y_{M,N} + \z^T\z/M$
	    \ENDFOR
	    \STATE $Y_{M,N} \gets (Y_{M,N})^{1/p}$
    \end{algorithmic}
\end{algorithm}        
        
\subsection{Error Analysis}   
Given a Chebyshev polynomial approximation $\psi_N(x)$ to $x^{p/2}$ over the spectrum of $\A$ we define the following estimator
\begin{equation}\label{eqn:ymn}
    Y_{M,N} = \left(\frac{1}{M}\sum_{j = 1}^M \w_j^T\phi_N(\A)\w_j\right)^{1/p} 
\end{equation}
where $\phi_N(\A) = [\psi_N(\A)]^2.$ Note that, by construction, $\phi_N(\A)$ is SPSD matrix. We now extend the analysis in Section~\ref{sec: mc}. 
\begin{proposition}\label{prop: YMN properties}
    Let $Y_{M,N}$ be defined as in~\eqref{eqn:ymn}. For fixed $N$, we have
    \begin{enumerate}
        \item (Non-negative): $Y_{M,N} \geq 0$ for all $M$; \label{eq: YMN nonneg}
        \item (Almost Sure Convergence:) $\ds\lim_{M \to \infty}Y_{M,N} = \left(\tr\left(\phi_N\left(\A\right)\right)\right)^{1/p}$ a.s.;\label{eq: YMN a.s}
        \item (Expectation): $\Ev(Y_{M,N}) \leq \left(\tr\left(\phi_N\left(\A\right)\right)\right)^{1/p}$;\label{eq: YMN Ev}
        \item (Variance): $\Var(Y_{M,N}) \leq \ds\frac{2\|\phi_N(\A)\|_F^2}{M\left(\tr\left(\phi_N\left(\A\right)\right)\right)^{2-2/p}}$.\label{eq: YMN Var}
    \end{enumerate}

\end{proposition}
\begin{proof}
     Note that since $\phi_N(\A)$ is SPSD, the non-negativity of $Y_{M,N}$ is immediate. Furthermore, applying the results of Propositions \ref{prop: X_M a.s.}, \ref{proposition: X_M EV}, and \ref{prop: X_M Var} to $\phi_N(\A)$, one can derive properties (\ref{eq: YMN a.s}), (\ref{eq: YMN Ev}), and (\ref{eq: YMN Var}) respectively. We omit the details. \hfill
\end{proof}

In this next result, we derive a bound on the smallest degree of the Chebyshev polynomial to ensure a user-defined relative error in the Schatten-p estimator.

\begin{proposition}\label{prop:ymnrelerr}
    Let $0  < \varepsilon \leq 1$, $p \geq 1$, $q = p/2$, and $\kappa = \sqrt{\frac{b}{a}}$. If the degree of the Chebyshev polynomial, $N$, satisfies
    \begin{equation}\label{eq: N bound}
    N \geq \ds\frac{\log\left(\ds\frac{4}{\varepsilon}(\kappa^2 +1)^q(\kappa - 1)\left(\kappa^p + \sqrt{\frac{\varepsilon}{2} + \kappa^{2p}}\right)\right)}{\log\left(\ds\frac{\kappa+1}{\kappa-1}\right)},
  \end{equation}
 then $|\tr(\phi_N(\A)) - \normP{\A}^p| \leq \ds\frac{\varepsilon}{2}\normP{\A}^p$.
\end{proposition}
\begin{proof}
The proof follows a similar strategy to~\cite[Theorem 3.1]{han2017approximating} and has several steps.
\paragraph{Error in terms of Chebyshev polynomials} The absolute error in $\normP{\A}$ can be bounded using the approximation properties of the Chebyshev polynomials.
\BEA
    \left|\normP{\A}^p - \tr(\phi_N(\A))\right| & = & \left|\ds\sum_{j=1}^n \lambda_j^p - \ds\sum_{j=1}^n \phi_N(\lambda_j)\right|\\
  & \leq & \sum_{j=1}^n \left|\lambda_j^p - \phi_N(\lambda_j)\right|\\
  & \leq & \max_{1 \leq j \leq n} n\left|\lambda_j^p - \phi_N(\lambda_j)\right|\\
   & \leq & n\max_{x \in [a,b]} \left|x^p - \phi_N(x)\right|.
\EEA
Since $\phi_N(x) = \psi_N^2(x)$, by repeated use of the triangle inequality
\BEA
    \left|x^p - \psi_N^2(x)\right| & = & \left|x^{2q} + x^q\psi_N(x) - x^q\psi_N(x) + \psi_N^2(x)\right|\\
    & \leq &
    \left|x^q\right|\left|x^q - \psi_N(x)\right| + \left|\psi_N(x)\right|\left|x^q - \psi_N(x)\right|\\
    & \leq & 2\left|x^q\right|\left|x^q - \psi_N(x)\right| + \left|x^q - \psi_N(x)\right|^2.
\EEA 
In the second step, we wrote $|\psi_N(x)| = |\psi_N(x) - x^q + x^q|$ and applied the triangle inequality.

\paragraph{Chebyshev polynomial approximation} Let ellipse $E$ in the complex plane with foci at $\pm 1$  and passing through the point $\left(\ds\frac{b + a}{b-a}, 0\right)$. The sum of major and minor semi-axes, denoted by $\rho > 1$, can be computed as 
\[     \rho = \frac{b + a}{b - a} + \sqrt{\left(\frac{b + a}{b - a}\right)^2 -1} = \frac{\sqrt{b} + \sqrt{a}}{\sqrt{b} - \sqrt{a}} = \frac{\kappa + 1}{\kappa - 1}
\]
where $\kappa = \sqrt{\frac{b}{a}}$ was defined in the statement of the proposition.

From~\cite[Corollary 2.2]{han2017approximating}, since $g(x) = x^{q}$ is analytic on the inside of the ellipse $E$, we have 
\[
    \max_{x \in [a,b]}\left|x^q - \psi_N(x)\right| \leq \frac{4U}{(\rho -1)\rho^N}.
\]
where the scalar $U$ satisfies
$$U = \ds\max_{z \in E} \left|g\left(\frac{b-a}{2}z + \frac{b+a}{2}\right)\right| = (b+a)^q .$$ 

\paragraph{Converting absolute error into relative error} Therefore, by the first two steps, 
\begin{equation}\label{eq: max Err}
    \max_{x \in [a,b]}|x^p - \phi_N(x)| \leq \left(2b^q + \frac{4U}{(\rho - 1)\rho^N}\right)\frac{4U}{(\rho - 1)\rho^N}. %
\end{equation}
We want to find $N$ such that $\ds\max_{x \in [a,b]}|x^p - \phi_N(x)| \leq \varepsilon a^p/2$. If such an $N$ can be found, then
\[ |\tr(\phi_N(\A)) - \normP{\A}^p| \leq \ds n \max_{x \in [a,b]}|x^p - \phi_N(x)| \leq \ds \frac{n\varepsilon a^p }{2} \leq \frac{\varepsilon}{2} \normP{\A}^p, \]
as desired. We now show that such an $N$ can be found. 

\paragraph{Solving for $N$} To this end, consider
\BEA
    \frac{\varepsilon a^p}{2} & \geq & \left(2b^q + \frac{4U}{(\rho - 1)\rho^N}\right)\frac{4U}{(\rho - 1)\rho^N} \\
    & = & \left(\frac{4U}{(\rho - 1)\rho^N}\right)^2 + 2b^q\frac{4U}{(\rho - 1)\rho^N} + b^{2q} - b^{2q}\\
    & = & \left(\frac{4U}{(\rho - 1)\rho^N} + b^q\right)^2 - b^{2q}.
\EEA 
Simplifying this expression, we get 
$$ \rho^N \geq
    \ds\frac{4U}{(\rho - 1)\left(\sqrt{\ds\frac{\varepsilon a^p}{2} + b^{p}} -b^q \right)}.$$
We have the elementary identity $$\ds\frac{1}{\sqrt{x+d} - \sqrt{x}} \ds\frac{\sqrt{x+d} + \sqrt{x}}{\sqrt{x+d} + \sqrt{x}} = \ds\frac{\sqrt{x+d} + \sqrt{x}}{d}$$ for all $x,d \geq 0$. Applying this inequality with $x= b^q$ and $d = \varepsilon a^p/2$, we get 
\[
    \rho^N \geq
    \ds\frac{4U}{(\rho - 1)\left(\sqrt{\ds\frac{\varepsilon a^p}{2} + b^{p}} -b^q \right)} = \ds\frac{8U\left( b^q + \sqrt{\frac{\varepsilon a^p}{2} + b^p}\right)}{(\rho - 1)\left(\varepsilon a^p\right)}.
\]
Since $\rho > 1$,  $N$ is bounded from below as
\begin{equation}\label{eq: N gen bound}
    N \geq \ds\frac{1}{\log(\rho)} \log\left(\ds\frac{8U\left( b^q + \sqrt{\frac{\varepsilon a^p}{2} + b^p}\right)}{(\rho - 1)\left(\varepsilon a^p \right)}\right)%
\end{equation}
Substitute the expressions for $U$ and $\rho$ into \eqref{eq: N gen bound} and simplify to get \eqref{eq: N bound}.\hfill
\end{proof}

In Figure~\ref{fig: Nbound}, the bound in~\eqref{eq: N bound} is plotted with  various values of $p$ and $\kappa \in [1,2]$ and $\varepsilon = 0.1$. Here a dot is placed when the value of $N$ is larger than $q = p/2$, suggesting that the bound is pessimistic for condition numbers larger than 2. However, using arguments in Newman and Rivlin~\cite[Theorem 2]{newman1976approximation}, we conjecture that $N = O(\sqrt{q})$ should be sufficient to accurately approximate $\A^p$. 
Similarly, one can use a low-degree rational approximation to accurately approximate $\A^p$; see~\cite{nakatsukasa2018rational} for additional details.

\begin{figure}[!ht]
    \centering
    \includegraphics[width = .75\textwidth]{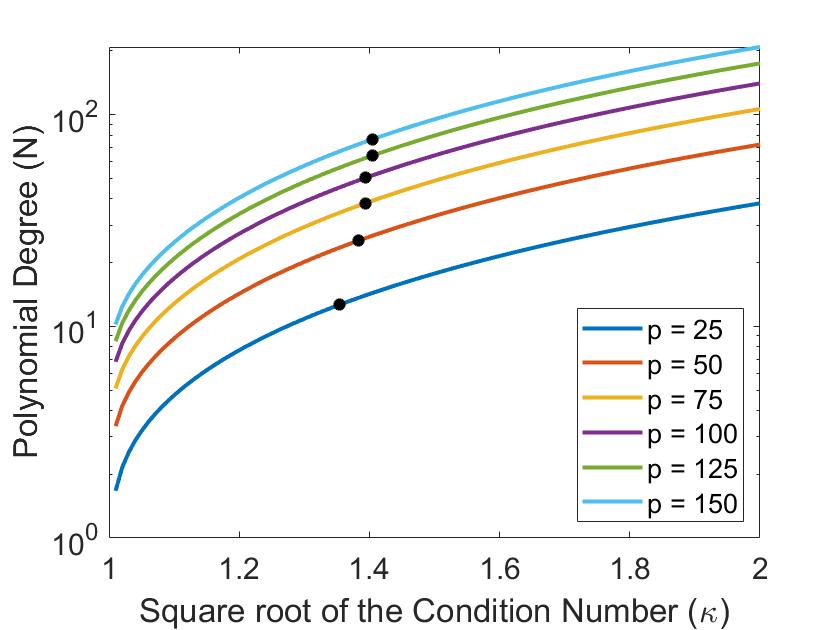}
    \caption{Plotting the values of \eqref{eq: N bound} for $p = 25, 50, 75, 100, 125$ and 150 with $\kappa \in [1,2]$ and $\varepsilon = 0.1$. This corresponds to approximating $\A^{q}$, where $p=2q$ and $\A$ has a condition number between 1 and 4. Here a black dot has been placed to mark the first instance that $N > q$.}
    \label{fig: Nbound}
\end{figure}

Although Proposition~\ref{prop:ymnrelerr} guarantees that $\tr(\phi_N(\A))$ has a small relative error, it is computationally challenging to implement since it involves constructing $\phi_N(\A)$ explicitly. Therefore, we approximate its trace using the Monte Carlo estimator $Y_{M,N}$. Using this proposition, we derive bounds for the absolute error in the $L^1$ sense (i.e., the bias) and the number of samples required for an $(\varepsilon,\delta)$-estimator for $\normP{\A}$.

\begin{theorem}\label{theorem: YMN eps delta}
Consider the same setup as in~\ref{prop:ymnrelerr}. Let $Y_{M,N}$ be defined as in~\eqref{eqn:ymn} and let $N$ satisfy \eqref{eq: N bound}, then
\begin{enumerate}
    \item ($L^1$ bound): $\left| \Ev\left(Y_{M,N}\right) - \normP{\A}\right| \leq (1 + \frac{\varepsilon}{2})\normP{\A}\left(\ds\frac{2}{M}\right)^{1/2} + \frac{\varepsilon}{2}\normP{\A}$\label{eq: YMN L1 bound}
    \item (($\varepsilon, \delta$) estimator): if $M \geq 72\varepsilon^{-2}\ln\left(\frac{2}{\delta}\right)$ then $Y_{M,N}$ is an $(\varepsilon, \delta)$ estimator for $\normP{\A}$.
\end{enumerate}
\end{theorem}
\begin{proof}
First consider the $L^1$ bound on $Y_{M,N}$. From Proposition~\ref{prop:ymnrelerr}, we have 
\[ \left(1-\frac{\varepsilon}{2}\right)\normP{\A}^p \leq \tr(\phi_N(\A)) \leq\left(1 + \frac{\varepsilon}{2}\right)\normP{\A}^p.\]
From the simple identity $(1-x)^p \leq (1-x) \leq (1+x) \leq (1+x)^p$ for $0 \leq x \leq 1$ and $p \geq 1$, we get 
\[ \left(1-\frac{\varepsilon}{2}\right)\normP{\A} \leq (\tr(\phi_N(\A))^{1/p} \leq \left(1 + \frac{\varepsilon}{2}\right)\normP{\A},\]
or $|(\tr(\phi_N(\A))^{1/p} - \normP{\A}| \leq \frac{\varepsilon}{2}\normP{\A}$. 

By the triangle inequality and by applying Proposition~\ref{prop: Bias bound} to $\phi_N(\A)$,  we find
\BEA 
    \left|\Ev\left(Y_{M,N}\right) - \normP{\A}\right| & \leq & \left|\Ev\left(Y_{M,N}\right) - \left(\tr\left(\phi_N\left(\A\right)\right)\right)^{1/p}\right| + \left|\left(\tr\left(\phi_N\left(\A\right)\right)\right)^{1/p} - \normP{\A}\right|\\
   & \leq & \left(\frac{2}{M}\right)^{1/2}\left(\tr\left(\phi_N\left(\A\right)\right)\right)^{1/p} + \left|\left(\tr\left(\phi_N\left(\A\right)\right)\right)^{1/p} - \normP{\A}\right| \\
    & \leq & \left(1 + \frac{\varepsilon}{2}\right)\normP{\A}\left(\ds\frac{2}{M}\right)^{1/2} + \frac{\varepsilon}{2}\normP{\A}.
\EEA 
If $M \geq 72\varepsilon^2\ln\left(\ds\frac{2}{\delta}\right)$ then, by~\cite[Theorem 3]{roosta2015improved},
\[
    \Pr\left(\left|Y_{M,N}^p - \tr(\phi_N(\A))\right| \leq \frac{\varepsilon}{3}\tr(\phi_N(\A))\right) \geq 1 - \delta.
\]
Furthermore as $\varepsilon \in (0,1)$, then $\tr(\phi_N(\A)) \leq \left(1 + \frac{\varepsilon}{2}\right)\normP{\A}^p \leq \frac{3}{2}\normP{\A}^p.$ Then with probability at least $1-\delta$
\[
    \left|Y_{M,N}^p - \tr(\phi_N(\A))\right| \leq \frac{\varepsilon}{2}\normP{\A}^p.
\]
Using the triangle inequality, with the same probability 
\[ \begin{aligned} \left|Y_{M,N}^p - \normP{\A}^p\right| \leq& \> \left|Y_{M,N}^p - \tr(\phi_N(\A))\right| + |\tr(\phi_N(\A)) - \normP{\A}^p| \\ \leq  & \>\frac{\varepsilon}{2}\normP{\A}^p + \frac{\varepsilon}{2}\normP{\A}^p = \varepsilon\normP{\A}^p. \end{aligned}\]
Now consider the following measurable sets
\[\begin{aligned}
    \calE = &\> \left\{\omega \in \Omega \Big| |Y_{M,N}^p(\omega) - \normP{\A}^p| \leq \varepsilon\normP{\A}^p\right\}\\
    \calD = &\> \left\{\omega \in \Omega \Big| |Y_{M,N}(\omega) - \normP{\A}| \leq \varepsilon\normP{\A}\right\}.
\end{aligned}
\]
Using a similar argument as in Proposition~\ref{theorem: ProbConc} we can show $1 - \delta \leq \Pr(\calE) \leq \Pr(\calD)$. Thus $Y_{M,N}$ is an $(\varepsilon, \delta)$ estimator for $\normP{A}$.\hfill
\end{proof}

Once again, we point out that the number of samples required for an $(\varepsilon,\delta)$ estimator for $\normP{A}$ is independent of the degree $p$, provided $N$ is sufficiently large.

\section{Numerical Experiments}\label{sec: results}
In this section, we will present numerical experiments demonstrating the performance of the estimators and the convergence analysis on several test matrices. The first set of test matrices are synthetically generated, the second set comes from the SuiteSparse collection, and the final test matrix arises from an application to Optimal Experimental Design (OED).
        
\subsection{Choices of Matrices}
\subsubsection{Synthetic Test Matrices}
For the matrix $\A \in \R^{100 \times 100}$ we constructed the following test matrices; the test problems are of the form $\A = \mathbf{QDQ}^T$, where $\mathbf{D}$ represent the eigenvalues taking particular values. The orthogonal matrix
$\mathbf{Q} \in \R^{100\times 100}$ is constructed by first generating a standard Gaussian random matrix, and then computing its QR factorization.  
      
\begin{enumerate}
    \item  \textbf{Linear Decay:} The first test matrix $\A_\text{linear} = \mathbf{Q}\D_\text{lin}\mathbf{Q}^T \in \R^{100 \times 100}$ has eigenvalues 
    \[ \D_\text{lin} = \text{diag}(6,7,\dots,105).\]
        \item \textbf{Clustered:} The second test matrix takes the form $\A_\text{clustered} = \mathbf{Q}\D_\text{clus}\mathbf{Q}^T$ where
        \[
           \D_\text{clus} = \text{diag}(\underbrace{100, \dots 100}_{20}, \underbrace{1, \dots , 1}_{80}).
        \]
    \item \textbf{Quadratic Decay:} The test matrix takes the form  $\A_{\text{quad}} = \mathbf{Q}\D_\text{quad}\mathbf{Q}^T$
        \[
           \D_{\text{quad}} = \text{diag}(1, 2^{-2}, \dots, 100^{-2}).
        \]
    \item  \textbf{Exponential Decay:} The test matrix takes the form $\A_{\text{exp}} = \mathbf{Q}\D_\text{exp}\mathbf{Q}^T$
        \[
            \A_{\text{exp}} = \text{diag}(0.9^1, \dots, 0.9^{100}).
        \]
\end{enumerate}  
       
The test matrices simulate different scenarios of eigenvalue distributions for an SPSD matrix. We have plotted the eigenvalue distributions in Figure~\ref{fig:: eigs} to illustrate these distributions.  
    \begin{figure}[t]
        \centering
        \includegraphics[scale = .4]{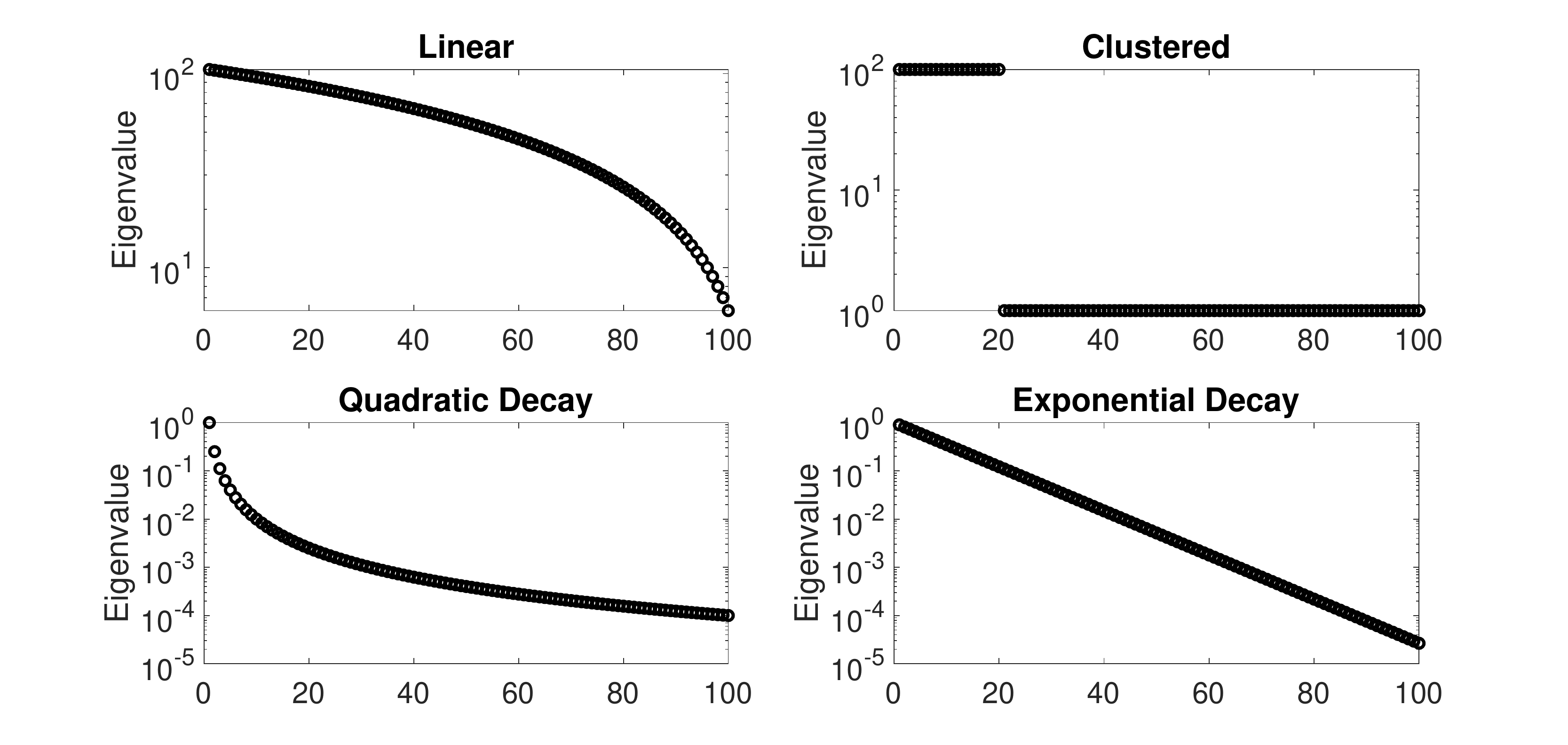}
        \caption{Semi-log plots of the eigenvalue distributions for each of the test matrices; each of size $100\times 100$.}
        \label{fig:: eigs}
    \end{figure}
        
\subsubsection{Test matrices from SuiteSparse collection}
    In addition to the test matrices described above, we also consider two relatively large matrices from the Suite Sparse matrix collection \cite{kolodziej2019suitesparse}. In particular, we choose
        \begin{enumerate}
            \item Trefethen$\_700$ matrix, a $700 \times 700$ SPD matrix from an application in combinatorics with a condition number approximately $4.71\times 10^3$, and %
            \item mhd4800b matrix, a $4800 \times 4800$ SPD matrix from an application in electrohydrodynamics with a condition number approximately $8.16\times 10^{13}$. %
        \end{enumerate} 
        
        \subsubsection{Application to Optimal Experimental Design}\label{ssec:oed}
        For the last test matrix, we return to our motivating problem from Optimal Experimental Design (OED).
        Our goal is to compute the Schatten $p$-norm of the posterior covariance operator, arising from a Bayesian linear inverse problem. 
        
        \newcommand{\fwd}{\mathbf{F}}
        \newcommand{\dpar}{\boldsymbol{\phi}}
        We consider the inverse problem of estimating the initial state in 
        the following 1D heat equation:
        \begin{equation}
            \begin{cases}\label{eq: PDE}
                u_t = ku_{xx} & x \in [0,1], \; t \in (0, t_f],\\
                u(x,0) = \phi(x) & x \in [0,1],\\
                u(0,t) = u(1, t) = 0 & t \in (0,t_f].
            \end{cases}
        \end{equation}
        Here $\phi(x)$ is an unknown initial state, which we seek to estimate using sensor measurements of the temperature at
        a few observation times. In~\eqref{eq: PDE},
        $k$ is the diffusion coefficient, which we choose to be
        $k = 2\times10^{-4}$.
        
        After discretization, the goal is to estimate the discretized
        parameter $\boldsymbol\phi$ from
        \begin{equation}\label{equ:inv}
            \fwd \dpar + \boldsymbol\eta = \mathbf{d}.
        \end{equation}
        Here $\fwd$ is the parameter-to-observable map, which maps
        the (discretized) initial state $\dpar$ to spatio-temporal observations, $\dpar$ is the discretized inversion parameter (the initial state), $\boldsymbol\eta$ is a random variable modeling 
        measurement noise, and $\mathbf{d}$ is measurement data.
        
        We discretize the problem~\eqref{eq: PDE} using
        finite-differences in space and implicit Euler in time. 
        Thus, an application of $\fwd$ to a vector requires solving~\eqref{eq: PDE}, and extracting solution values at the measurement points and at measurement times. Here we take measurements at 17 equally spaced
        sensors in the spatial domain $[0, 1]$ and at observation times
        $\{0.25, 0.5, 0.75, 1\}$. We assume that $\boldsymbol\eta$ in
        \eqref{equ:inv} is multivariate Gaussian with mean zero and
        covariance given by $\sigma \mathbf{I}$ with $\sigma = 0.002$ (corresponding to 0.1\% noise).
        
        We consider a Bayesian formulation~\cite{tarantola2005inverse} of this inverse problem. Assuming a Gaussian prior, we also have a Gaussian posterior. One possible measure of the posterior uncertainty is given by the Schatten $p$-norm of the posterior covariance operator. The latter is given by the following $254 \times 254$ matrix
        \[
           \postcov = (\sigma^{-2} \mathbf{F}^T \mathbf{F} + \prcov^{-1})^{-1}.
        \]
        Here $\prcov$ is the covariance operator of the Gaussian prior.
        For this example, we chose $\prcov = (\gamma \mathbf{K})^{-1}$ where $\gamma = 10^{-4}$ is a regularization parameter and $\mathbf{K}$ is the discretized Laplacian operator (with zero Dirichlet boundary conditions). In the numerical experiments in Section~\ref{sec: MC results} and \ref{sec: Cheby results}, we examine
        the effectiveness of our proposed estimators for computing $\normP{\postcov}$.

        \subsection{Monte Carlo Estimator Results}\label{sec: MC results}
            For each choice of test matrix described in previous subsection, we apply Algorithm~\ref{algorithm:: mc} and compute the error statistics as a function of the sample size $M$. For each fixed sample size $M$, we generated $500$ different realizations of $X_M$ using Algorithm~\ref{algorithm:: mc} and then found the average, $97.5$th quantile and $2.5$th quantile of the relative errors. 
            Note that the interval between the 2.5th and 97.5th quantiles is the same as the central 95th confidence interval for the error. 
            
            \paragraph{Synthetic test Matrices}
            In Figures~\ref{fig: CompareErrLowPTest} and \ref{fig: CompareErrHighPTest}, we display the mean and central 95th confidence interval for each of the four synthetic test matrices, using a value of $p = 5$ and $p = 120$ respectively.
            We call the shaded region within the 95th confidence interval the error envelope for $X_M$. First, we observe that the error statistics for the Monte Carlo estimator $X_M$ did not depend significantly on the eigenvalue distributions. Second, we observe that for $p=120$, the average relative error was lower compared to the average relative error for $p=5$ for all four eigenvalue distributions. 

            Furthermore, the error envelopes appear tighter suggesting smaller empirical variance with increasing $p$. This suggests the $(\varepsilon, \delta)$ bound for $X_M$ in Theorem~\ref{theorem: ProbConc} should decrease with increasing $p$.%
        \begin{figure}[!ht]
            \centering
                \begin{subfigure}
                    \centering
                    \includegraphics[width=.45\textwidth]{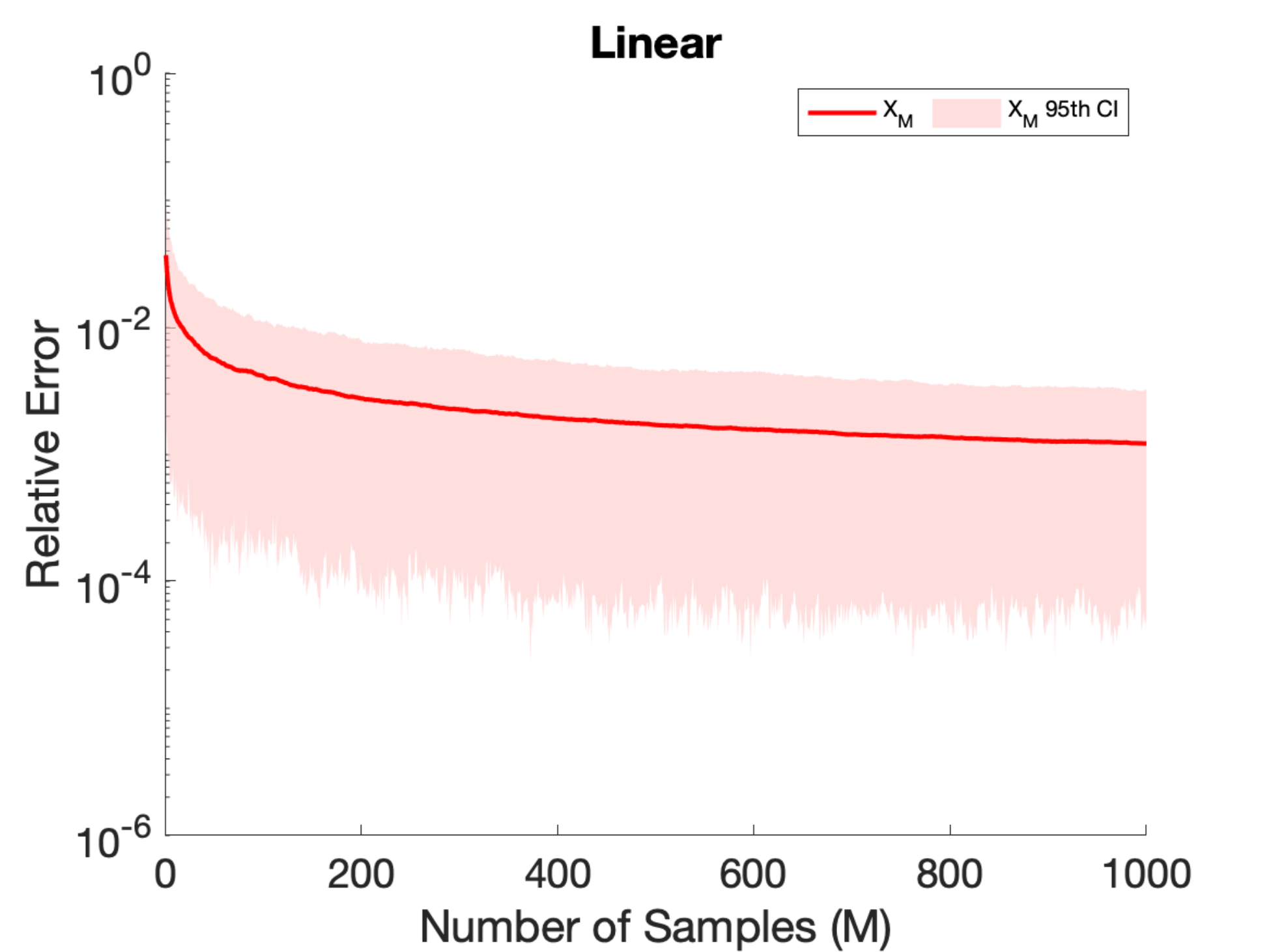}
                \end{subfigure}
                \begin{subfigure}
                    \centering
                    \includegraphics[width=.45\textwidth]{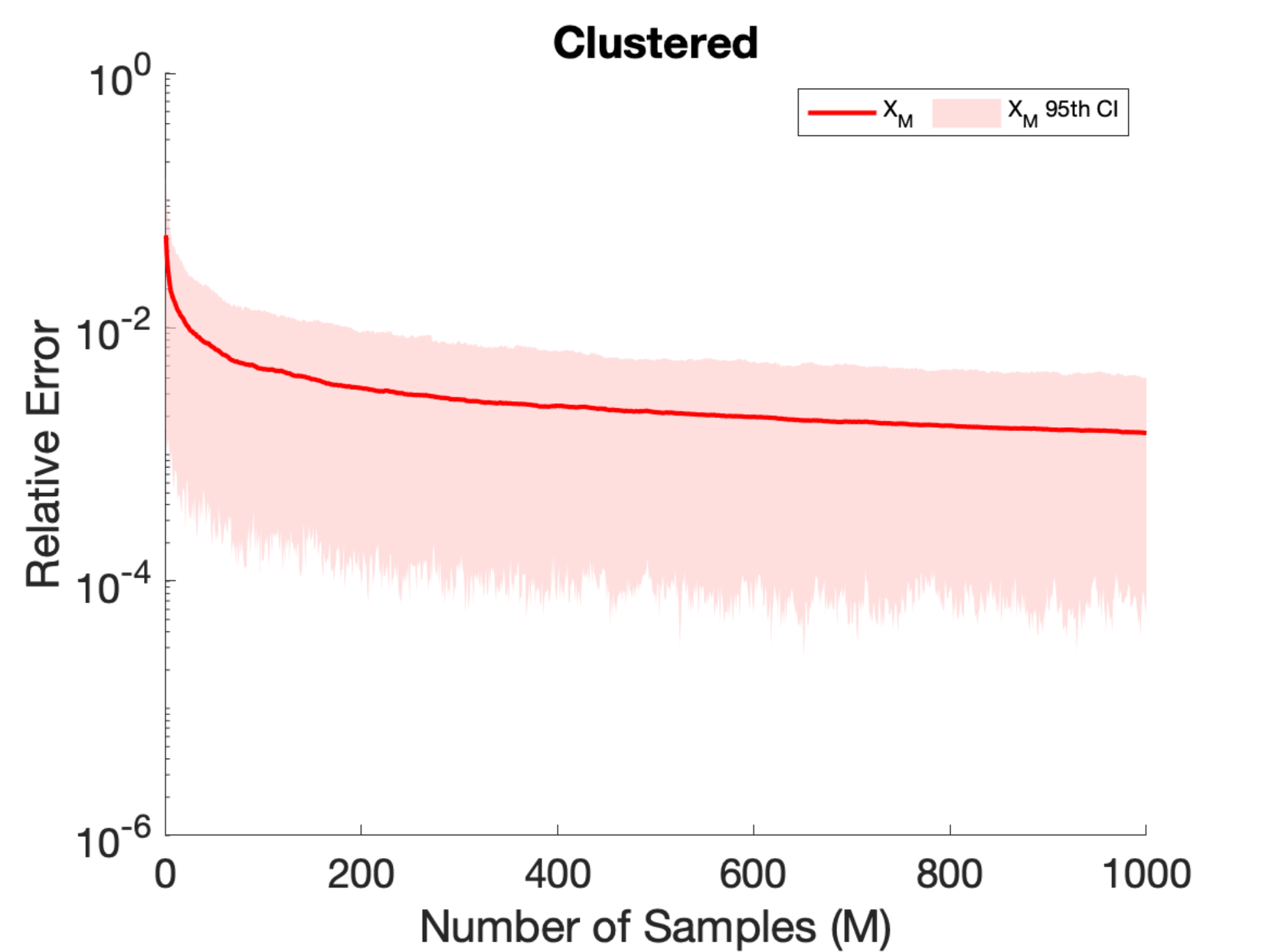}
                \end{subfigure}
                \begin{subfigure}
                    \centering
                    \includegraphics[width=.45\textwidth]{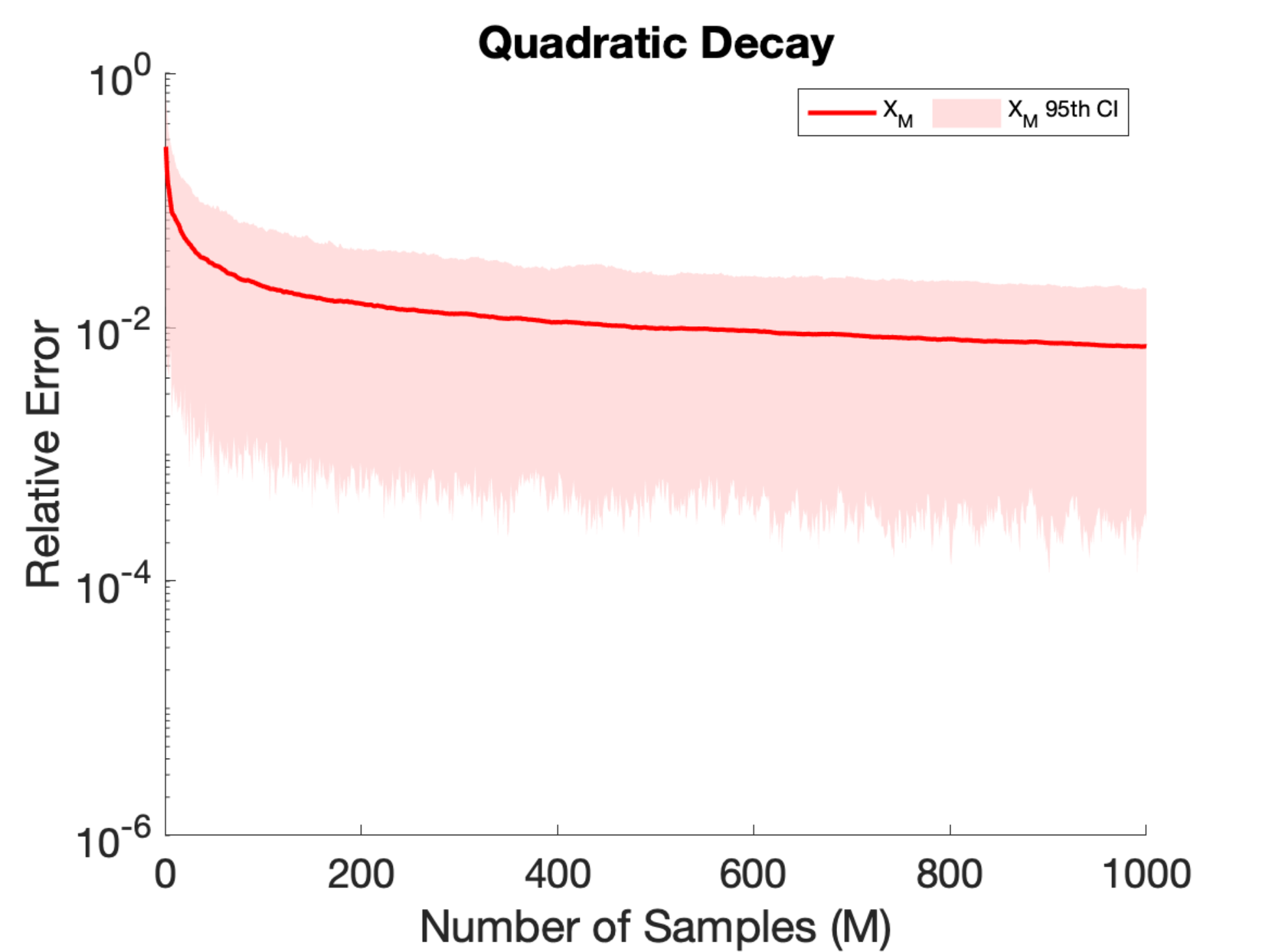}
                \end{subfigure}
                \begin{subfigure}
                    \centering                        
                    \includegraphics[width=.45\textwidth]{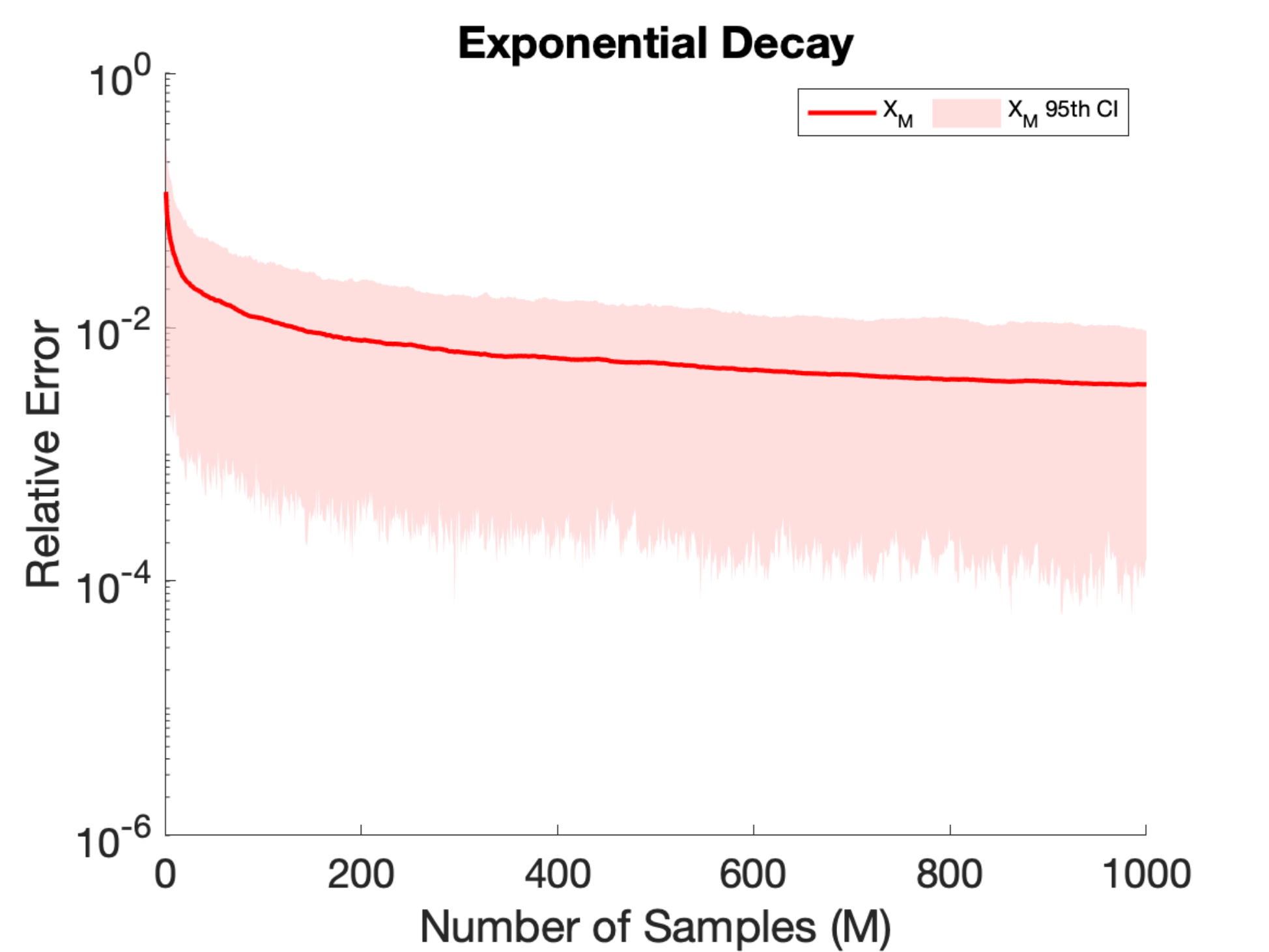}
                \end{subfigure}
            \caption{The relative error in the Monte Carlo Estimator $X_M$ for the 100 $\times$ 100 synthetic test matrices with $p = 5$. The error statistics were generated based on $500$ realizations each for a fixed sample size $M$.} %
            \label{fig: CompareErrLowPTest}
        \end{figure}
            
        \begin{figure}[!ht]
            \centering
                \begin{subfigure}
                    \centering
                    \includegraphics[width=.45\textwidth]{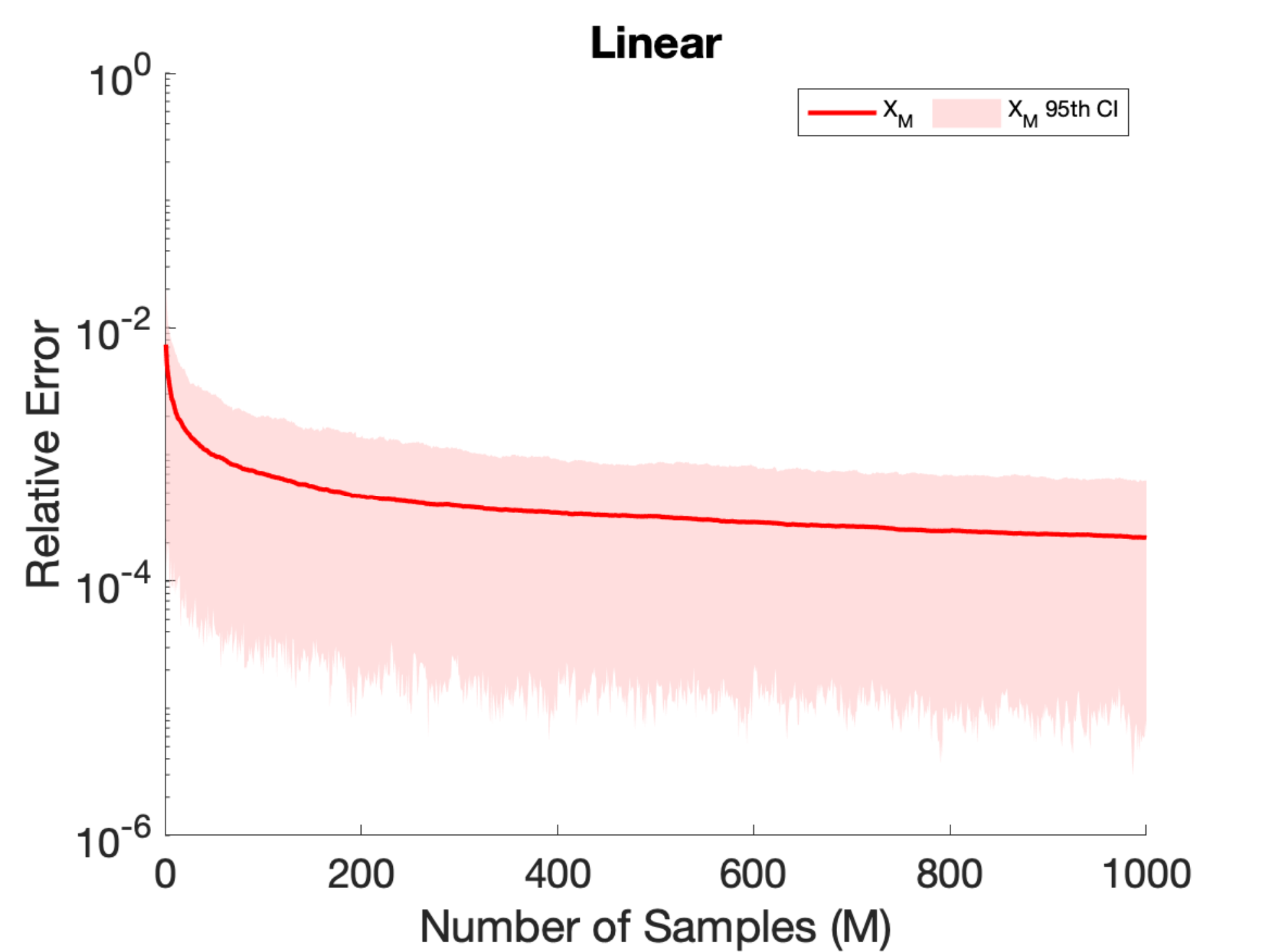}
                \end{subfigure}
                \begin{subfigure}                            
                    \centering
                    \includegraphics[width=.45\textwidth]{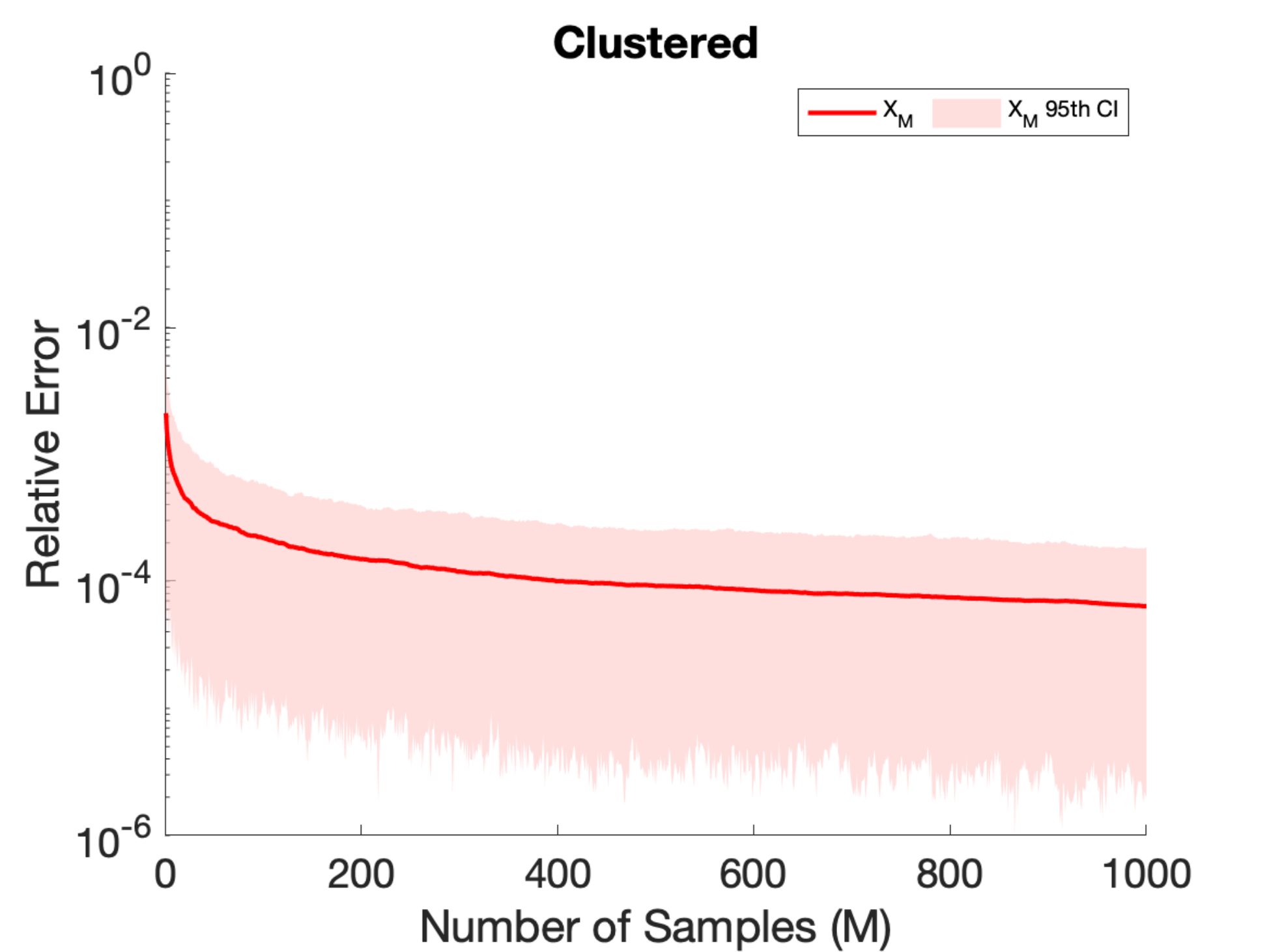}
                \end{subfigure}
                \begin{subfigure}
                    \centering
                    \includegraphics[width=.45\textwidth]{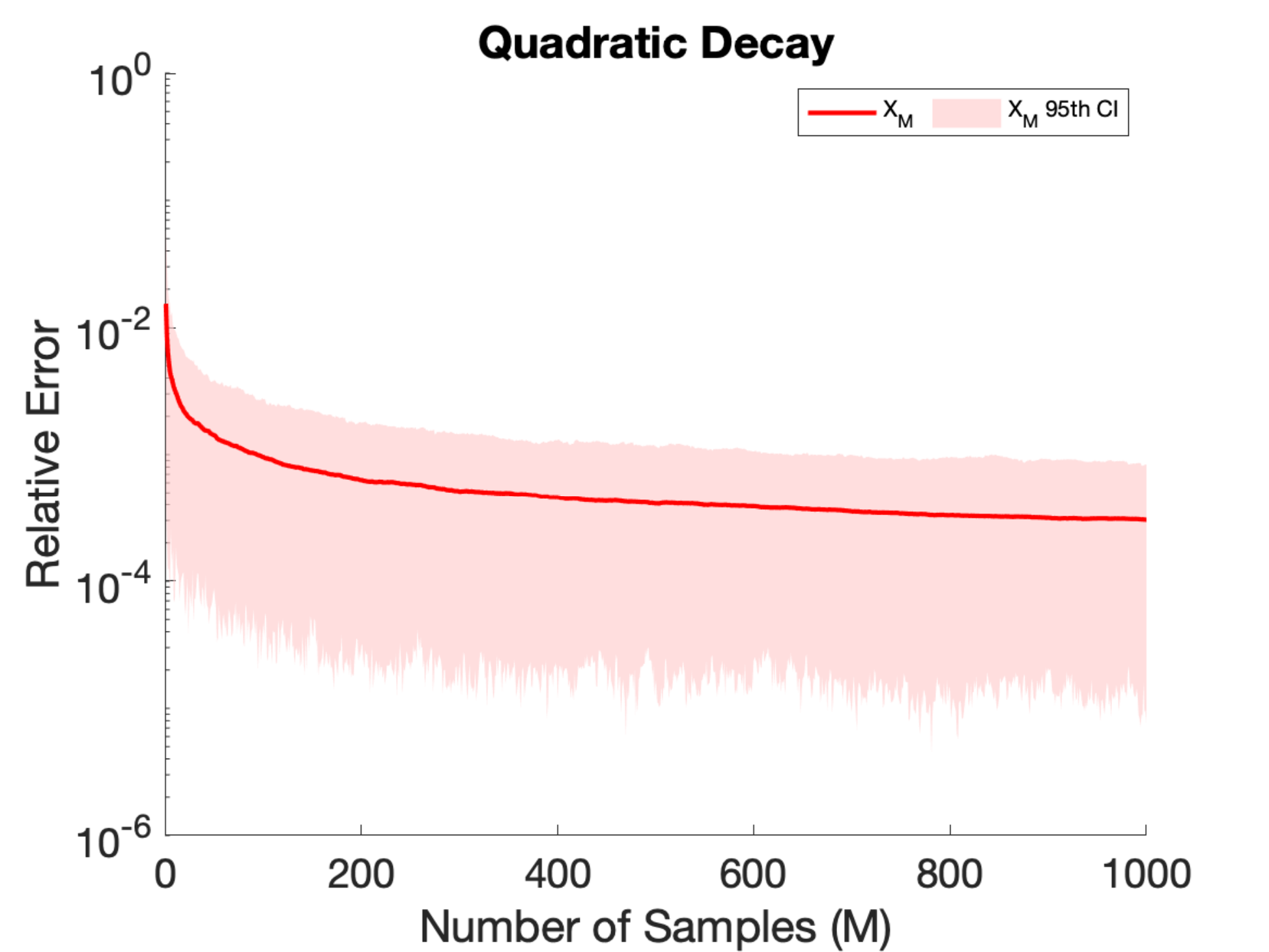}
                \end{subfigure}
                \begin{subfigure}
                    \centering
                    \includegraphics[width=.45\textwidth]{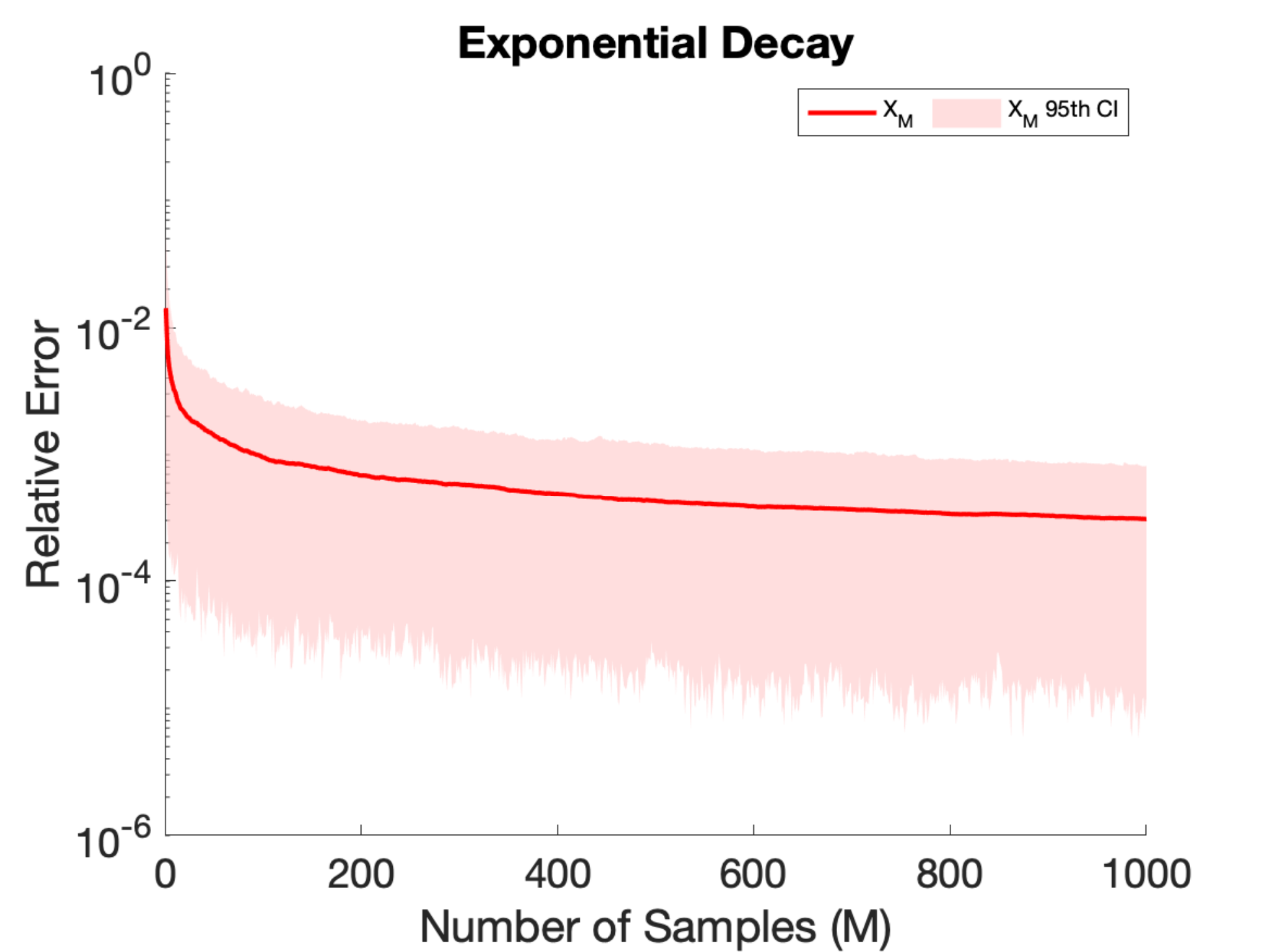}
                \end{subfigure}
            \caption{The relative error in $X_M$ for the 100 $\times$ 100 test matrices with $p = 120$. The error statistics were generated based on $500$ realizations each for a fixed sample size $M$.} %
            \label{fig: CompareErrHighPTest}
        \end{figure}
        
        \paragraph{Sparse Suite Matrices}
            We consider the two test matrices from the Suite Sparse matrix collection. 
            In Figure~\ref{fig: CompareErrLowPSparse} we display the error envelope when $p = 5$ and in Figure~\ref{fig: CompareErrHighPSparse} we plot the error envelope when $p = 80$. Once again the mean relative error decreased and the error envelope appears to tighten as $p$ increased. This further provides evidence that the relative error does not show strong dependency on the eigenvalue distribution. %
        
        \begin{figure}[!ht]
            \centering
                \begin{subfigure}
                    \centering
                    \includegraphics[width=.45\textwidth]{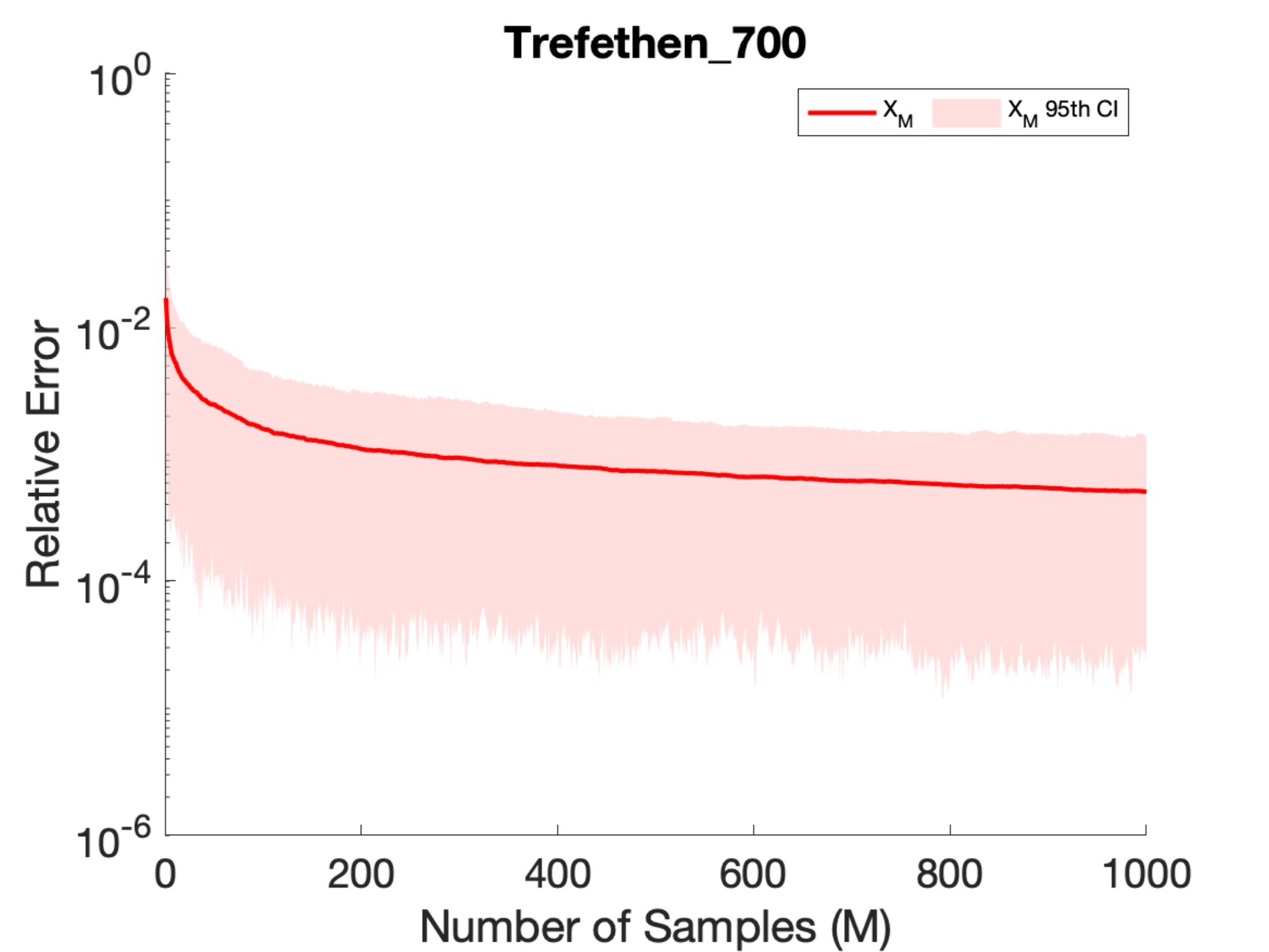}
                \end{subfigure}
                \begin{subfigure}
                    \centering
                    \includegraphics[width=.45\textwidth]{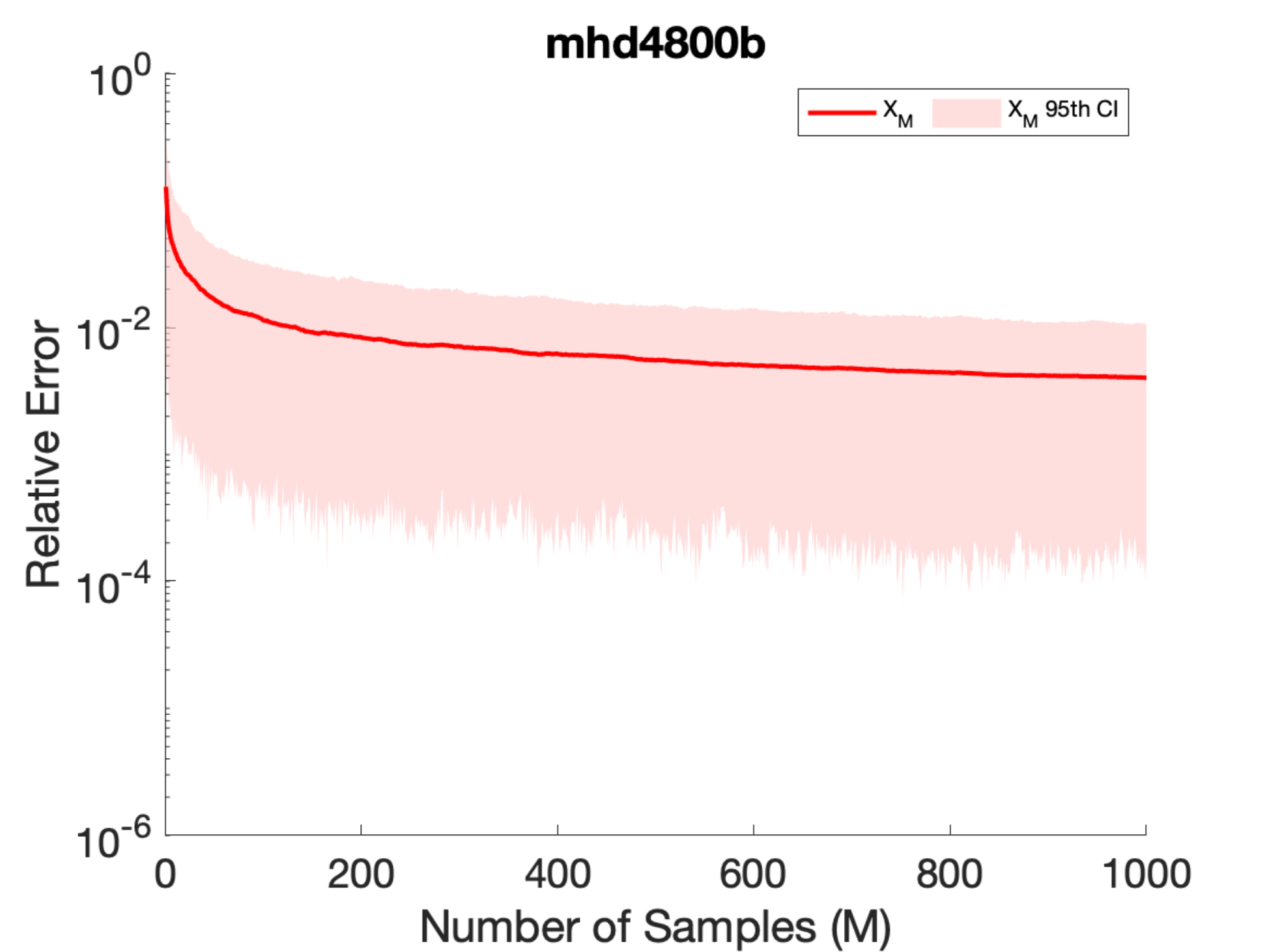}
                \end{subfigure}
            \caption{The relative error of $X_M$, with $p = 5$, for each of the matrices from the Suite Sparse matrix collection. The error statistics were generated based on $500$ realizations each for a fixed sample size $M$.} %
            \label{fig: CompareErrLowPSparse}
        \end{figure}
        \begin{figure}[!ht]
            \centering
                \centering
                \includegraphics[width=.45\textwidth]{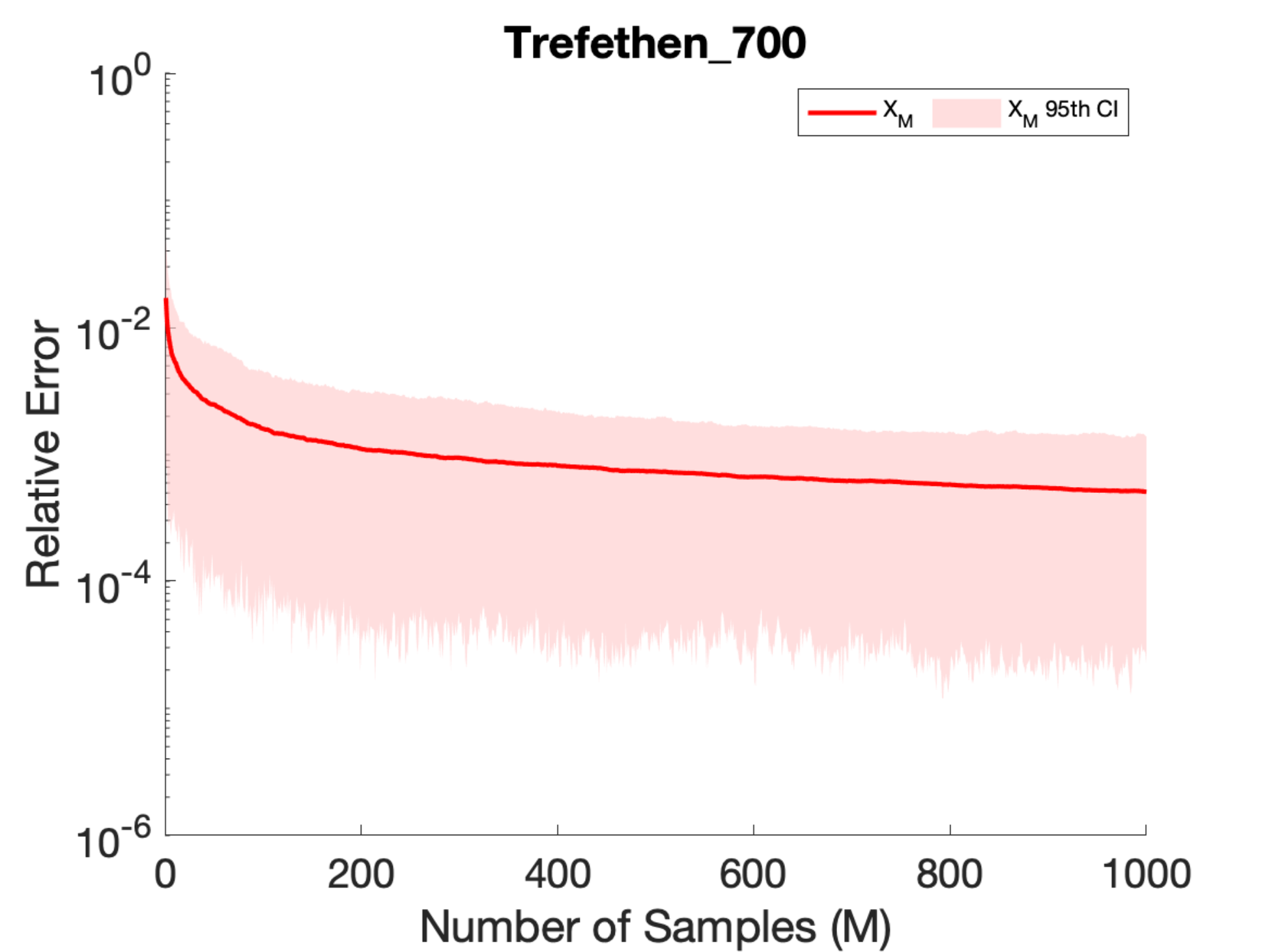}
                \includegraphics[width=.45\textwidth]{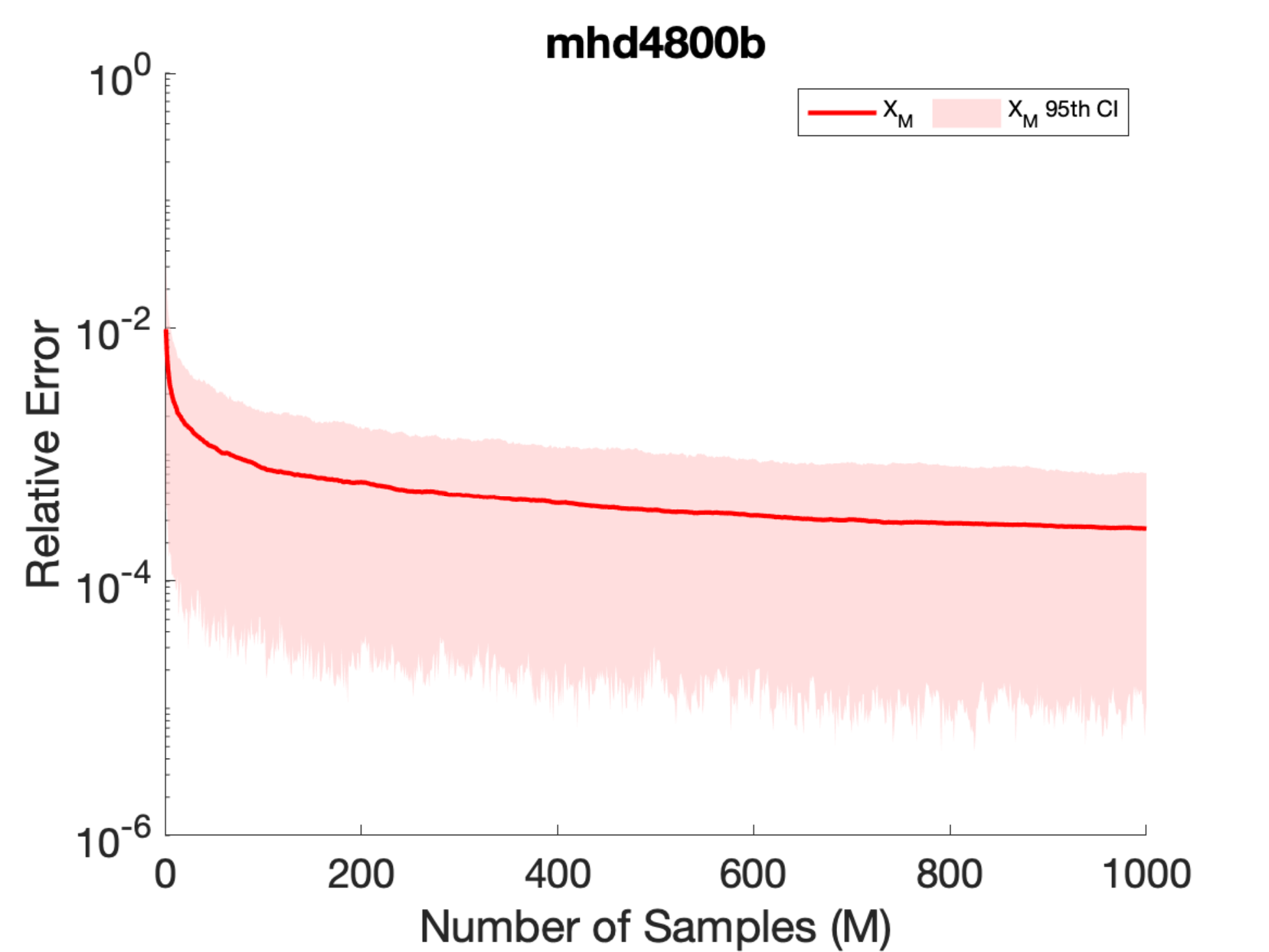}
            \caption{The relative error of $X_M$, with $p = 80$, for each of the matrices from the Suite Sparse matrix collection. The error statistics were generated based on $500$ realizations each for a fixed sample size $M$.} %
            \label{fig: CompareErrHighPSparse}
        \end{figure}
        
        \paragraph{Posterior Covariance Matrix}
            In Figure~\ref{fig: OED MC}, we display the relative error for the posterior covariance matrix generated using the setup in Section~\ref{ssec:oed}. The main conclusions from this plot are essentially the same as the other two sets of test matrices. %
            
       \begin{figure}[!ht]
            \centering
                \centering
                \includegraphics[width=.45\textwidth]{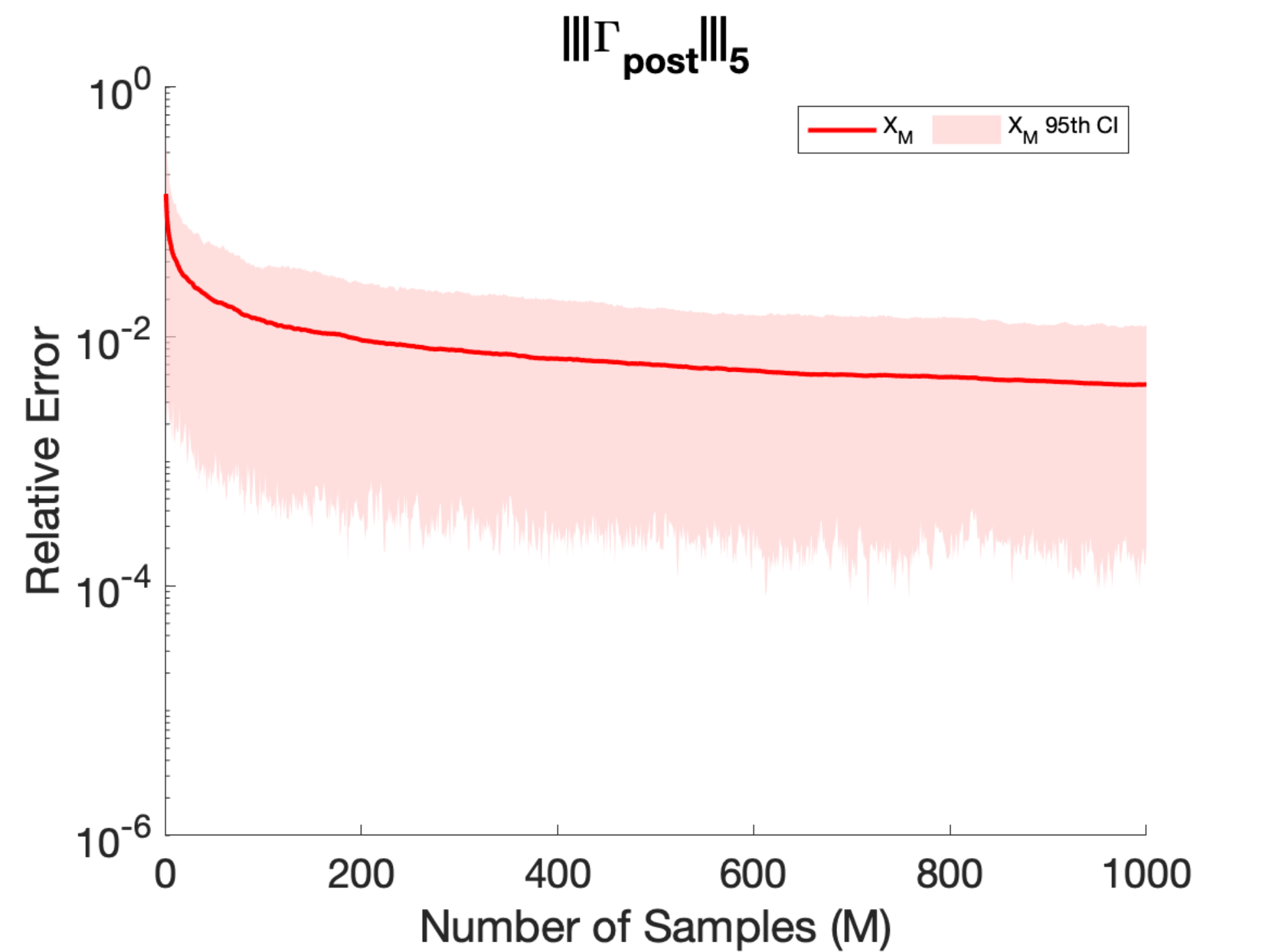}
                \includegraphics[width=.45\textwidth]{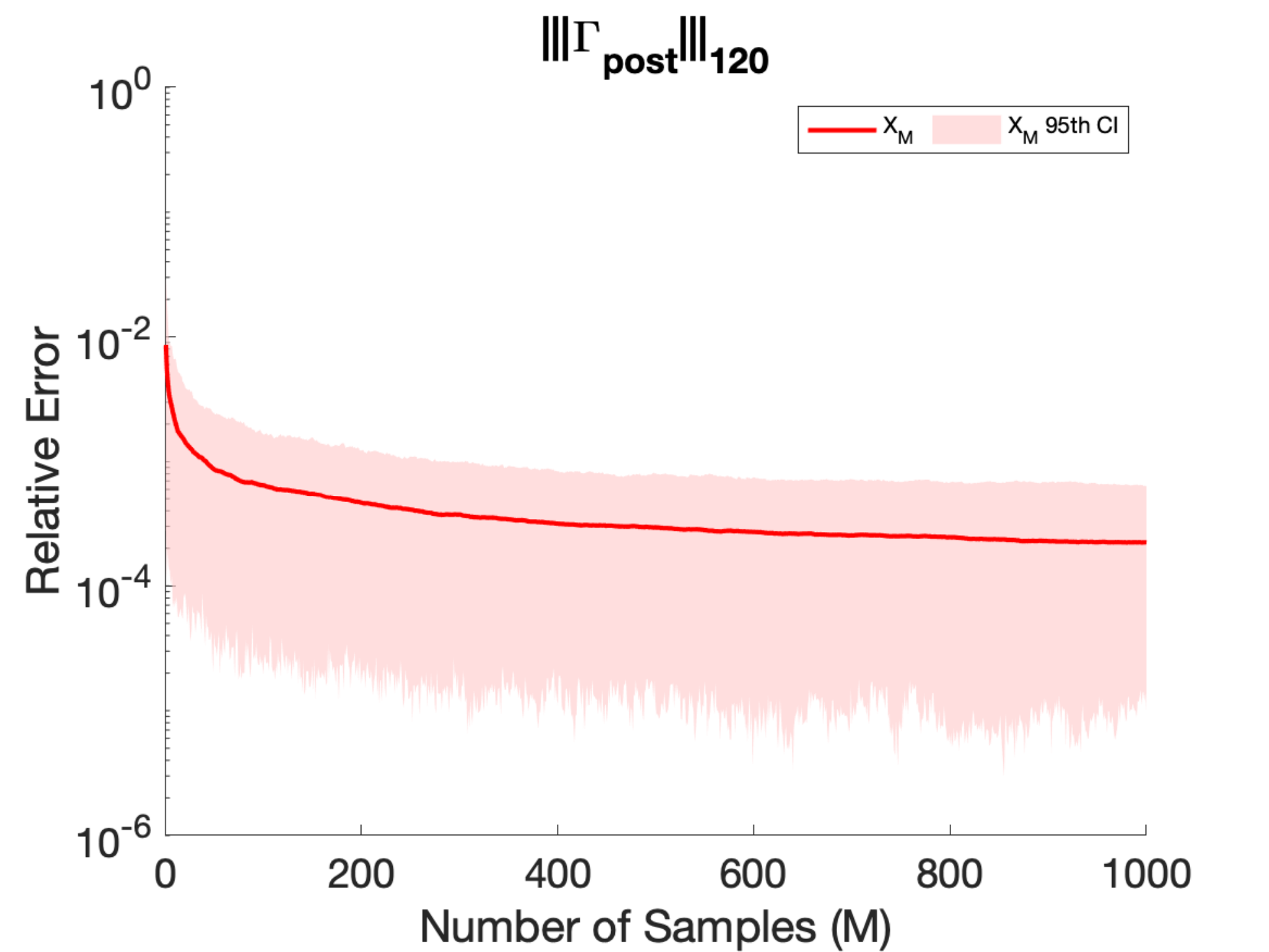}
            \caption{The relative error of $X_M$ for a 254 $\times$ 254 matrix $\Gamma_{post}$ with $p = 5$ and $p = 120$. Similar to the test matrices, we ran 500 different simulations for a fixed sample size $M$ and we computed the mean error and the 97.5th quantile and 2.5th quantile in the errors.}
            \label{fig: OED MC}
        \end{figure}
       
        \subsection{Chebyshev Monte Carlo Estimator}\label{sec: Cheby results}
            Recall that the bound on $N$ derived in Proposition~\ref{prop:ymnrelerr} was pessimistic. In this section we present numerical evidence that a relatively small $N$ is sufficient for accurately estimating $\normP{\A}$. %
            
            For the synthetic test matrices and the posterior covariance matrix, we chose $p =120$ and used $N = 5, 10, 20,  60$, whereas for the test matrices from the Suite Sparse collection, we used $p=80$ and $N = 5, 10, 20,  40$. For all the test matrices, we computed the error using Algorithm~\ref{algorithm:: Cheby}. Similar to the ``standard'' Monte Carlo method in Section~\ref{sec: MC results} we computed the average error by using $500$ realizations for a fixed sample size and value of $N$. %
            
            \paragraph{Synthetic Test Matrices}
            In Figure~\ref{fig: ChebyErrTest}, we display the mean relative error in $Y_{M,N}$ for $N = 5,10,20,30$ for each of the synthetic test matrices when $p = 120$. Notice that with $N = 20$ and $N = 30$ the average relative error has similar behavior as in Figures~\ref{fig: CompareErrHighPTest}. Next, we observe that the estimator $Y_{M,N}$ is accurate for all the test matrices here. However, if $N$ is small, i.e., $5-10$, then we see that increasing the number of samples does not decrease the average relative error due to the bias (i.e., error due to Chebyshev polynomial approximation). On the other hand, if $N$ is sufficiently large $\geq 10$, we see that increasing the sample size $M$ can reduce the average relative error.

            \paragraph{Suite Sparse Matrices}
            In Figure~\ref{fig: ChebySparse} we display the mean relative error for using Chebyshev approximation to accelerate the computation of $\normP{\A}$ for the Trefethen\_700 and mhd4800b matrices when $p = 80$. Here we used values of $N = 5, 10, 20, 30$ and notice similar trends as in the numerical experiments using synthetic matrices. For example, once again we observe that if $N$ is too low then increasing the sample size will not reduce the error in the Chebyshev approximation. Also we find that for both test matrices, $N = 20$ was sufficient for accurately approximating the Schatten $p$-norm.
            
            \begin{figure}[!ht]
            \centering
                \begin{subfigure}
                    \centering
                    \includegraphics[width=.45\textwidth]{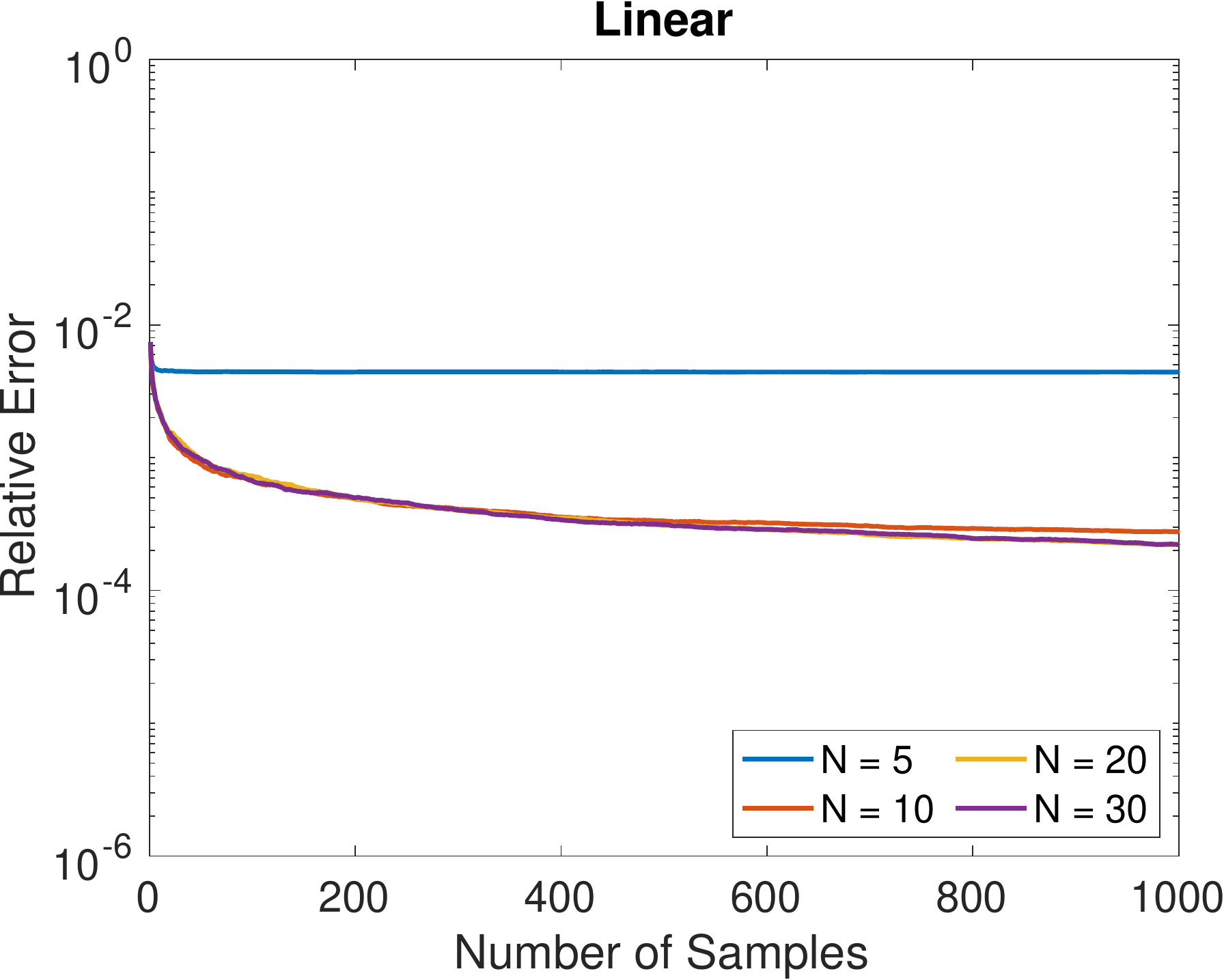}
                \end{subfigure}
                \begin{subfigure}
                    \centering
                    \includegraphics[width=.45\textwidth]{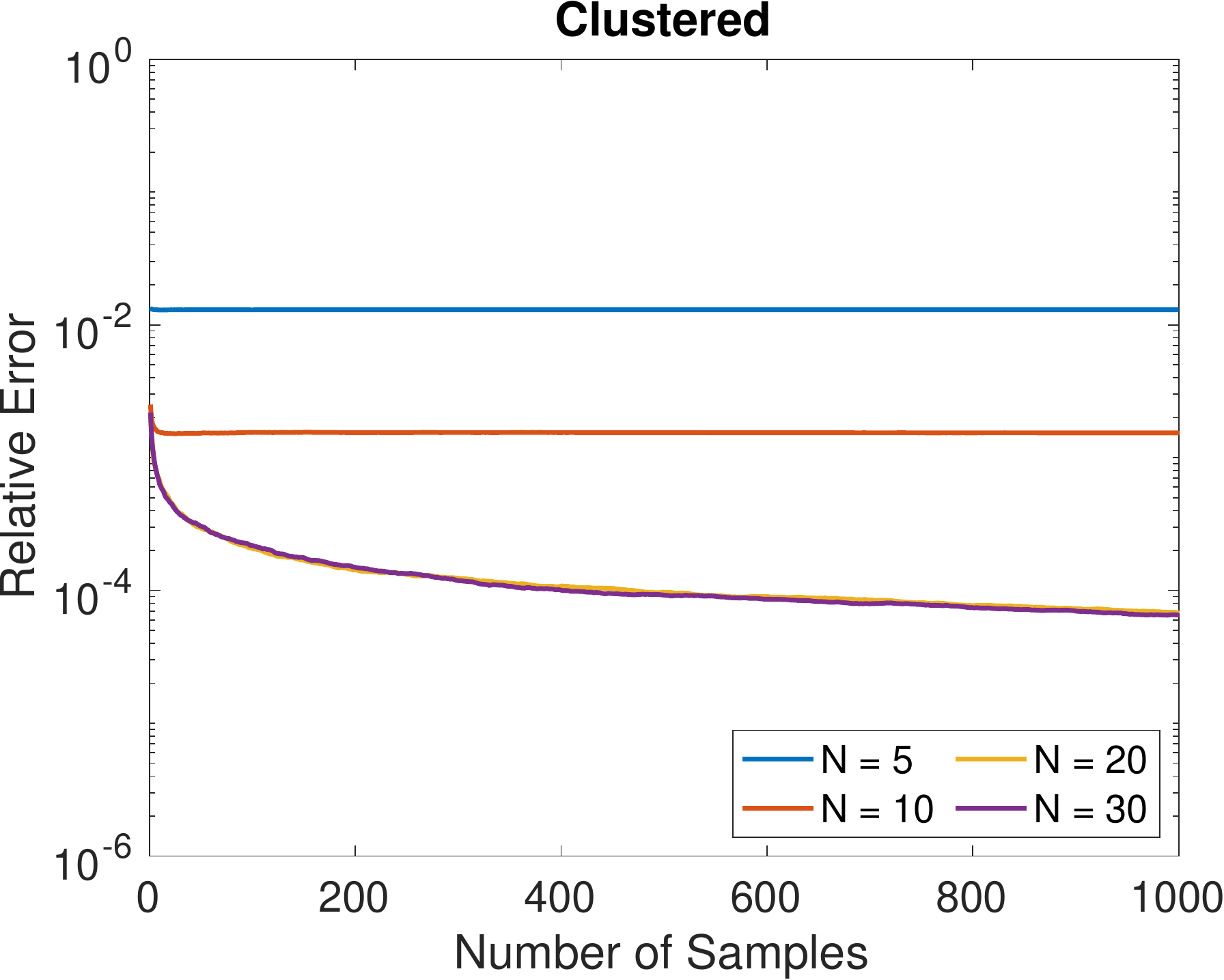}
                \end{subfigure}
                \begin{subfigure}
                    \centering
                    \includegraphics[width=.45\textwidth]{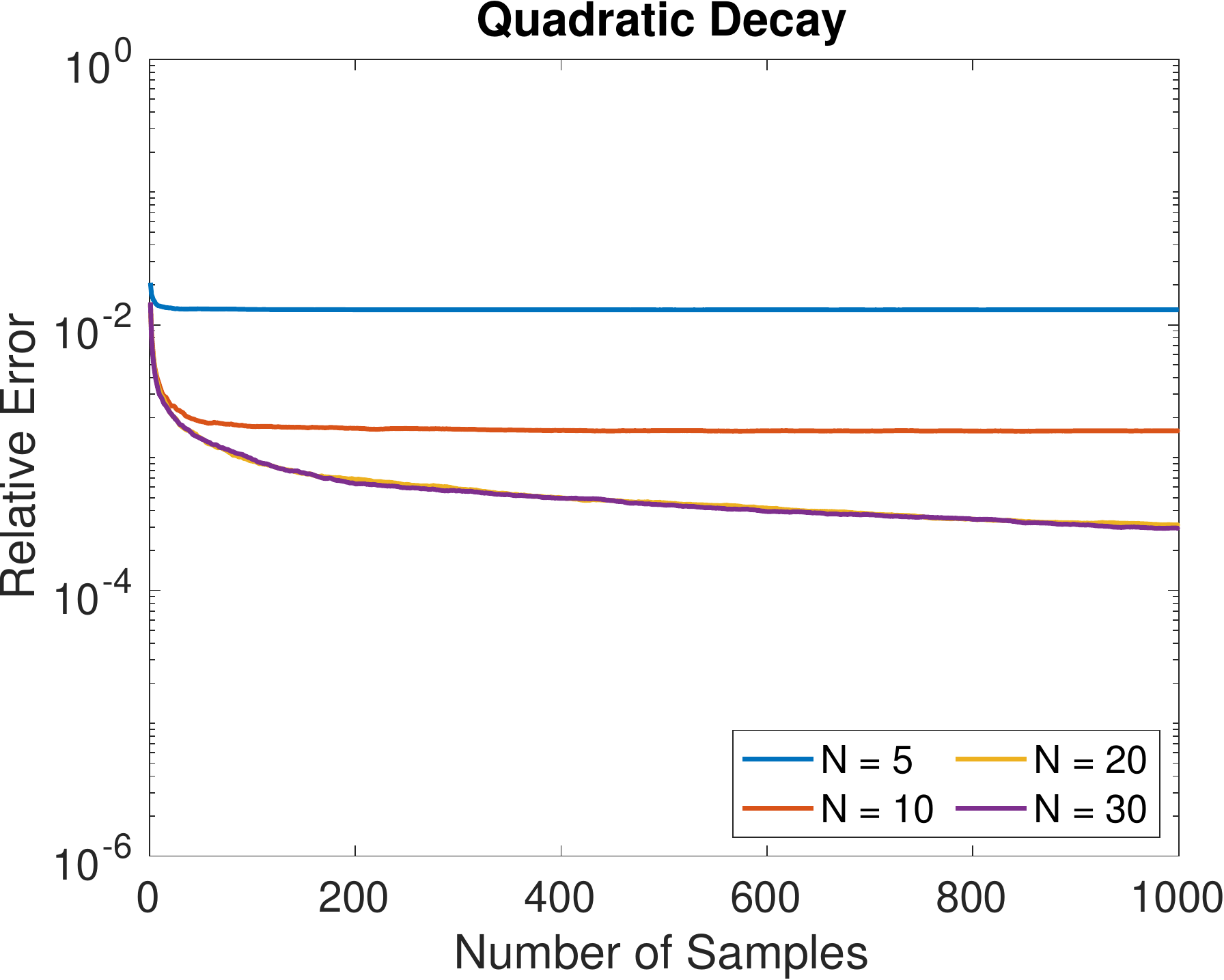}
                \end{subfigure}
                \begin{subfigure}
                    \centering
                    \includegraphics[width=.45\textwidth]{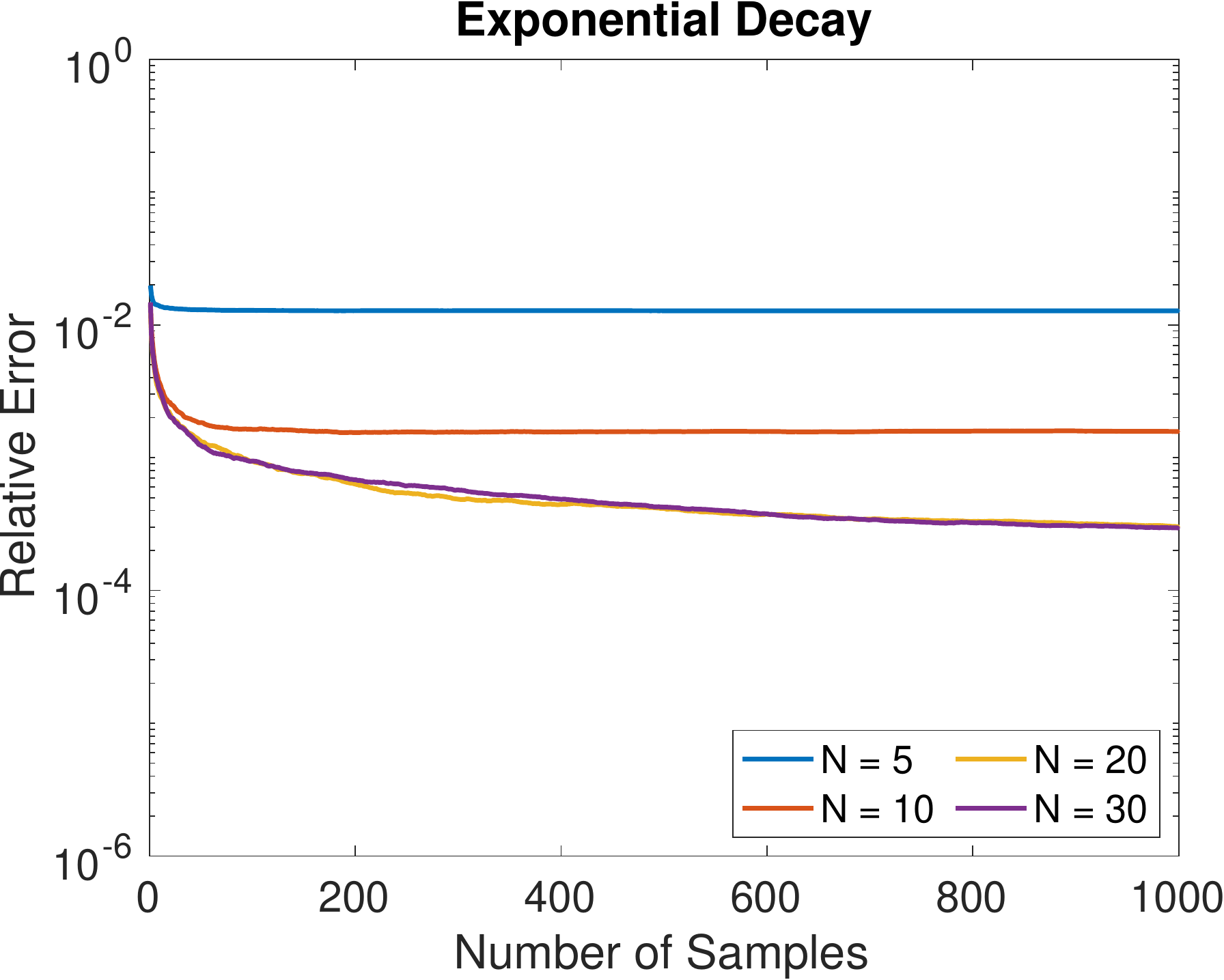}
                \end{subfigure}
                \caption{The average relative error of $Y_{M,N}$ for each of the 100 $\times$ 100 test matrices with $p = 120$. We used $500$ different realizations for a fixed sample size $M$ and degree $N$.}
                \label{fig: ChebyErrTest}
            \end{figure}
            \begin{figure}[!ht]
                \centering
                    \begin{subfigure}
                        \centering
                        \includegraphics[width=.45\textwidth]{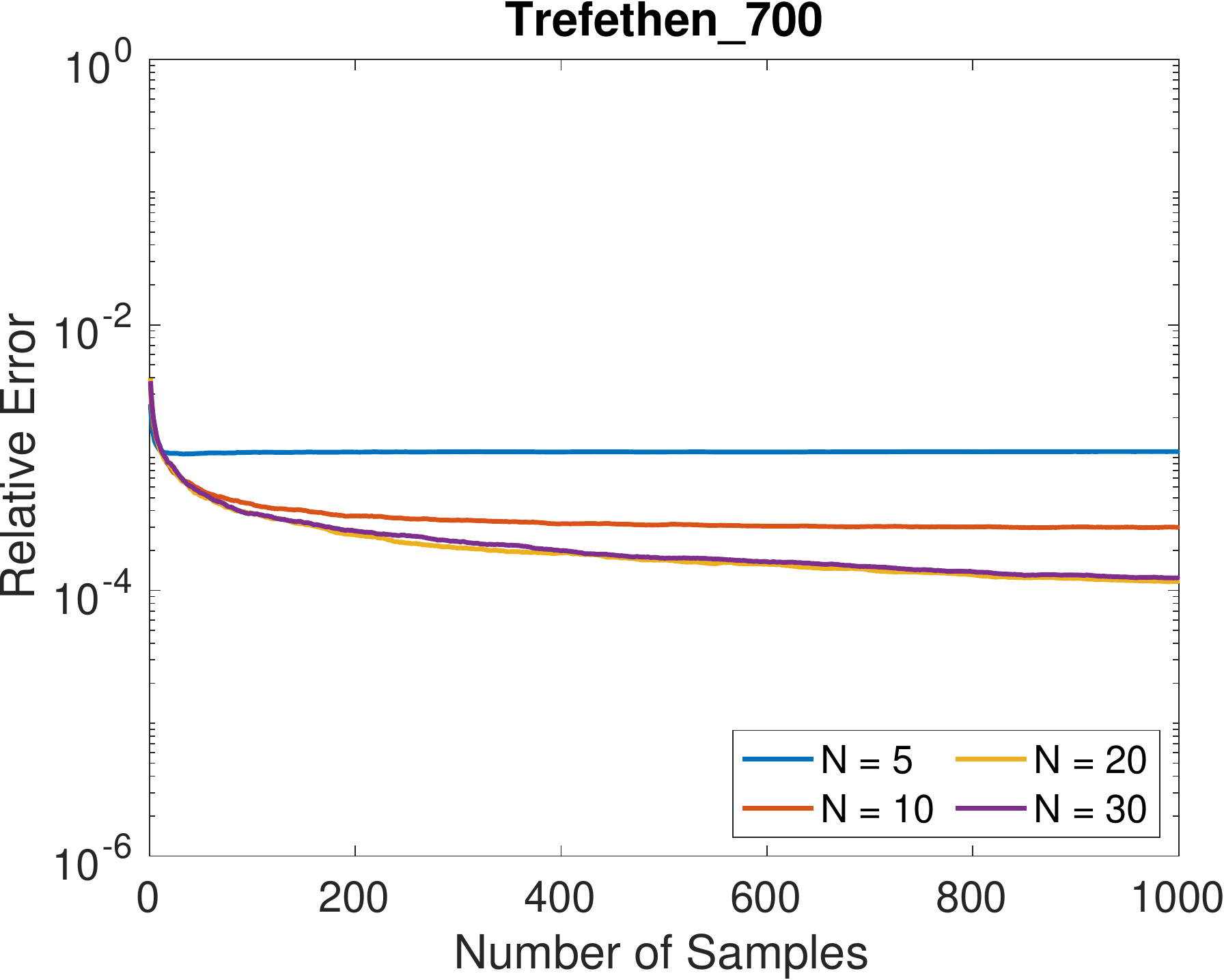}
                    \end{subfigure}
                    \begin{subfigure}
                        \centering
                        \includegraphics[width=.45\textwidth]{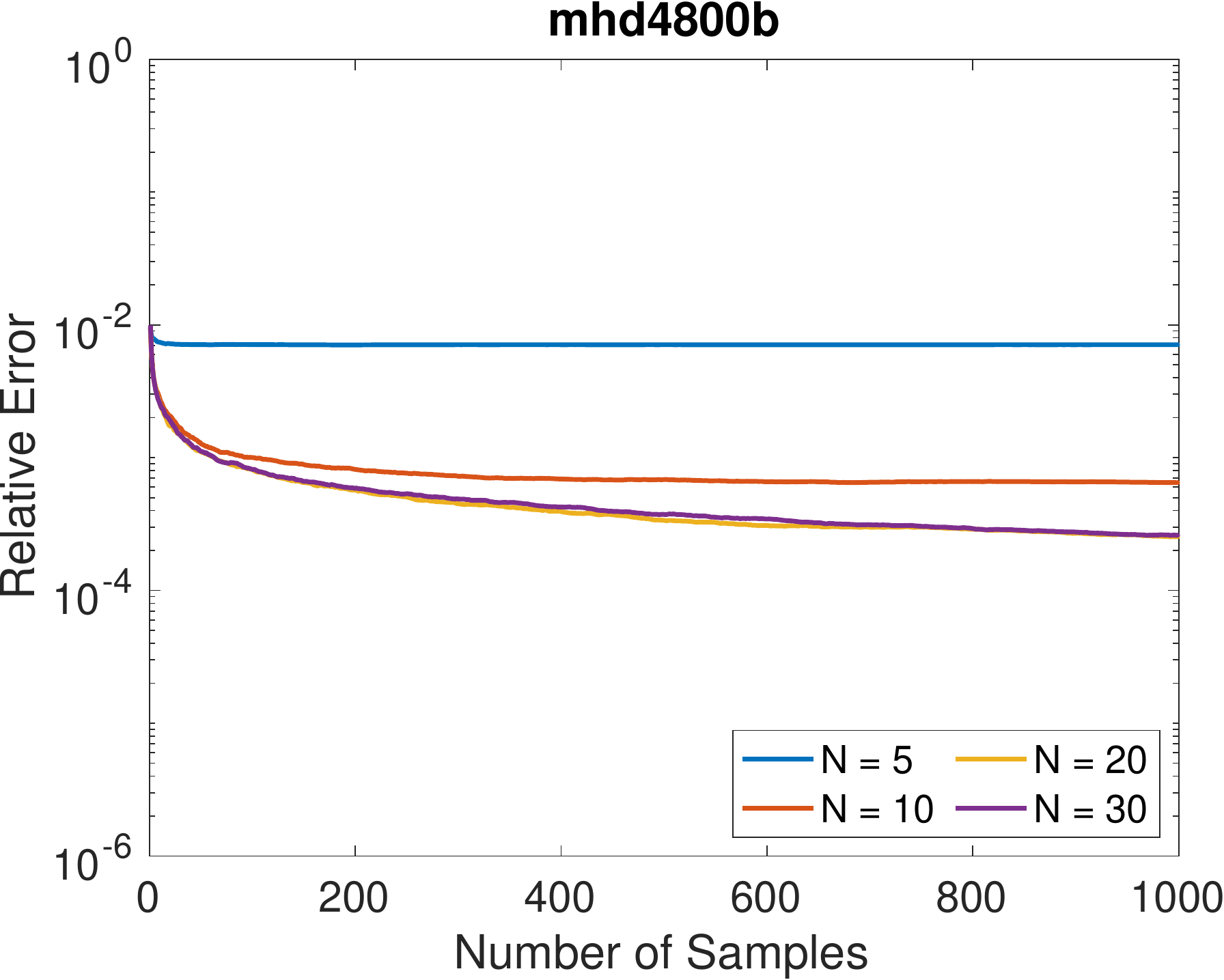}
                    \end{subfigure}
                \caption{The average relative error of $Y_{M,N}$ for each of the Suite Sparse matrices with $p = 80$. We used 500 different realizations for a fixed sample size $M$ and degree $N$ .}
                \label{fig: ChebySparse}
            \end{figure}
            
            \paragraph{Posterior Covariance Matrix}
            In Figure~\ref{fig: ChebyPost} we display the mean relative error in $Y_{M,N}$ for the Posterior Covariance Matrix from our OED example problem with $p=120$. Here we used $N = 5,10,20,30$ and, similar to the Test matrices, we find that $N = 20$ was sufficient to approximate $\normP{\Gamma_{post}}$, which is a speedup of a factor of $3$ in terms on number of matrix-vector products.  
            \begin{figure}[!ht]
                \centering
                \includegraphics[width=.45\textwidth]{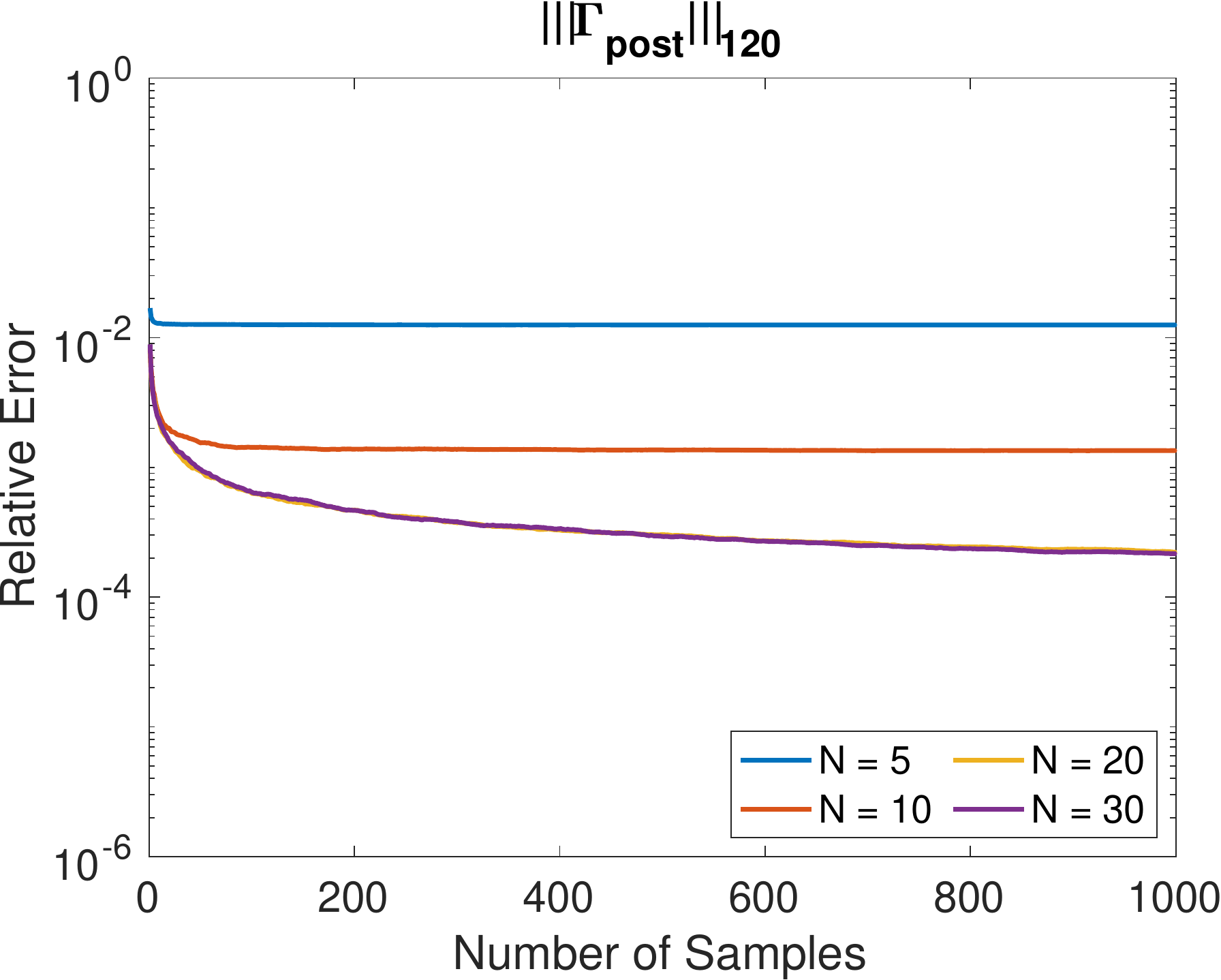}
                \caption{The average relative error of $Y_{M,N}$ for a 254 $\times$ 254 matrix $\Gamma_{post}$ with $p = 120$. We used 500 different realizations for a fixed sample size $M$ and degree $N$.}
                \label{fig: ChebyPost}
            \end{figure}

        We return to the question of the degree of Chebyshev polynomials. Numerical evidence suggested that $N = 20$ was sufficient for $p = 120$ and $N=10$ is sufficient for $p=80$ even with condition numbers as large as $8 \times 10^{13}$. This suggests that the bound in Proposition~\ref{prop:ymnrelerr} is pessimistic and suggests potential room for improvement.%

        Another point worth mentioning here is the trade-off between the degree of the polynomial and the number of samples used. If the degree of the polynomial is small, then even with a large number of samples, the error may be dominated by the bias in the Chebyshev polynomial approximation. On the other hand, if the degree of the polynomial is sufficiently high, then the error may be determined by the sample size. Suppose we are given a fixed computational budget for a certain number of matrix-vector products. For a given relative error, and a certain user defined probability, one can use Theorem~\ref{theorem: YMN eps delta} to give insight into apportioning the computational budget between the degree of the polynomial and the number of Monte Carlo samples. %

    \section{Conclusion}
    Computation of the Schatten $p$-norm is frequently used in linear algebra and analysis, however, computing it using a straightforward application of the definition can be  computationally difficult for large matrices. We proposed two different estimators and presented probabilistic analysis of their convergence and accuracy. The numerical results show that our estimators are efficient and accurate. They also serve to illustrate the main theoretical analysis developed in this paper, but show room for improvement. Specifically, we would like to show in Section~\ref{sec: mc}, the number of samples for an $(\varepsilon,\delta)$ estimator for $\normP{\A}$ decreases with $p$. Similarly, we would like to show that a small degree $N$ is sufficient for accurately estimating $\normP{\A}$ using $Y_{M,N}$. Other possible future directions involve using a stochastic Lanczos quadrature approach as in~\cite{ubaru2017fast}, which has the advantage that it does not require estimates of the extreme points of the spectrum and promises to be more accurate compared to the Chebyshev polynomial approximation. Another possible approach is using a rational approximation to $x^p$~\cite{ubaru2017fast}; while a relatively small degree rational function is sufficient, computing a rational matrix function can be computationally expensive.

    \section{Acknowledgements}
    We are grateful to Eric Hallman for his suggestion of using the symmetry of $\A$ to cut the operations cost in half in Algorithm \ref{algorithm:: mc}. The authors would like to acknowledge support from the National Science Foundation through the grant `` RTG: Randomized Numerical Analysis'' DMS - 1745654. %

    \bibliography{Schatten-pRef}
    \bibliographystyle{abbrv}
\end{document}